\documentclass{article}

\usepackage{enumitem}
\usepackage[colorlinks, linkcolor = black, citecolor = black, filecolor = black, urlcolor = blue]{hyperref}
\makeatletter
\def\namedlabel#1#2{\begingroup
  #2%
  \def\@currentlabel{#2}%
  \phantomsection\label{#1}\endgroup
}
\makeatother

\usepackage{MacrosRidges}

\def\balpha{{\bs{\alpha}}}
\def\bbeta{{\bs{\beta}}}
\def\bgamma{{\bs{\gamma}}}
\def\snhx{s_{n,h}^{}(\x)}
\def\shx{s_h^{}(\x)}
\def\bsnhx{\mathfrak{s}_{n,h}^{}(\x)}
\def\bshx{\mathfrak{s}_h^{}(\x)}
\def\bSnhx{\S_{n,h}^{}(\x)}
\def\bShx{\S_h^{}(\x)}

\def\bqnhx{\hat{\bs{q}}_{n,h}^{}(\x)}
\def\bQnhx{\hat{\bs{Q}}_{n,h}^{}(\x)}

\def\munhx{\hat{\bmu}_{n,h}^{}(\x)}
\def\Sigmanhx{\hat{\bSigma}_{n,h}^{}(\x)}

\def\Rddsym{\R^{d\times d}_{\mathtt{sym}}}

\def\Bvec{\vec{\B}}
\def\Gvec{\vec{\G}}
\def\Hvec{\vec{\H}}
\def\Lvec{\vec{\bs{L}}}
\def\Wvec{\vec{\W}}

\def\Zvec{\vec{\Z}}

\def\Bvechx{\vec{\B}_h^{}(\x)}
\def\Yvecnhx{\vec{\Y}_{\!\!n,h}^{}(\x)}

\usepackage{fullpage}
\usepackage{mathtools}
\usepackage{array}
\newcounter{overview}

\begin{document}
\title{Local Estimation of a Multivariate Density and its Derivatives\footnote{This work was supported by Swiss National Science Foundation. It is part of the first author's Phd dissertation.}}
\author{Christof Str\"ahl, Johanna F.~Ziegel and  Lutz D\"umbgen\\
	University of Bern, Switzerland}
\date{\today}
\maketitle

\begin{abstract}
We analyze four different approaches to estimate a multivariate probability density (or the $\log$-density) and its first and second order derivatives. Two methods, local log-likelihood and local Hyv\"arinen score estimation, are in terms of weighted scoring rules with local quadratic models. The other two approaches are matching of local moments and kernel density estimation. All estimators depend on a general kernel, and we use the Gaussian kernel to provide explicit examples. Asymptotic properties of the estimators are derived and compared. In terms of rates of convergence, a refined local moment matching estimator is the best.
\end{abstract}

\newpage

\setcounter{tocdepth}{2}
\tableofcontents

\section{Introduction}
\label{sec:introduction}

We consider independent random vectors $\X_1, \X_2, \ldots, \X_n$ with distribution $P$ given by a density function $f: \R^d \to [0, \infty)$. Our goal is to estimate the density function $f(\x)$, the gradient $Df(\x)$ and the Hessian matrix $D^2f(\x)$ at arbitrary points $\x \in \{f > 0\}$.

The gradient can be used for mode estimation. For instance, the mean shift
algorithm in \citet{Cheng1995} uses a gradient ascent method for
approaching the mode and for finding points where the gradient $D f$
is equal to zero. Together, the gradient and Hessian matrix can be used to
find density ridges, see for instance \citet{Genovese2014} and \citet{Polonik2016}. A point $\x$ is on an $s$-dimensional ridge whenever the $s$ smallest eigenvalues of the Hessian $D^2f(\x)$ are negative and the
corresponding eigenvectors are orthogonal to the Gradient $D f(\x)$; see
\citet{Genovese2014}. Modes can be seen as $0$-dimensional ridges. Ridges
are not only defined for density functions but for arbitrary functions on
$\R^d$ that are sufficiently differentiable.

The ridges of a function and of a strictly increasing transformation
thereof coincide; see \citet{OzertemETAL2011}. Therefore, in view of
estimating density ridges, one can consider estimates of $f$ or $\log f$.

The four estimation paradigms considered in this paper are local moment
matching (\texttt{M}), kernel density (\texttt{K}), local log-likelihood
(\texttt{L}) and local Hyv\"arinen score (\texttt{H}) estimation. Locality
is always achieved by using a kernel function, which has to satisfy some
high level assumptions that we state in Section
\ref{sec:local-estimation-f}. As an explicit example we consider the case
of the Gaussian kernel.

The local moment matching estimator (\texttt{M}) in Section \ref{sec:MM} matches empirical local moments with the Taylor expansion of their theoretical counterparts, where local means that the observations are weighted with a suitable kernel function. This approach has been introduced by \citet{Mueller2001} for rectangular kernels. In general, it leads to a system of linear equations and estimators with interesting asymptotic properties.

The kernel density estimator (\texttt{K}) in Section \ref{sec:KD} is the average of a kernel function recentered at the sample points $\X_1, \X_2,  \ldots, \X_n$, see \citet{WandJones1995} for a general introduction. As estimators for $Df$ and $D^2f$, the corresponding derivatives of the kernel density estimator are considered.

The other two estimators are both based on localizing proper scoring rules;
see \citet{HolzmannKlar2017}. Using the logarithmic score of
\citet{Good1952} leads to a local log-likelihood estimator (\texttt{L}). This
estimator was already considered by \citet{Loader1996}, and he proved
asymptotic results in the case of kernel functions with compact support. In
general, the estimator can only be calculated numerically. Using the score
of \citet{Hyvaerinen2005} leads to a new estimator (\texttt{H}) of the first and second derivatives of the log-density which is given by a system of linear equations. In case of the standard Gaussian kernel, this method coincides with the local log-likelihood method.

The remainder of this manuscript is structured as follows: The four estimation paradigms are introduced in Section~\ref{sec:local-estimation-f}. The asymptotic properties of the estimators (\texttt{M}), (\texttt{K}) and (\texttt{L}) are presented in Section~\ref{sec:asymptotics}. We show that the three paradigms are closely related and discuss their advantages and disadvantages.  We refrain from studying the asymptotics of the (\texttt{H}) estimator since it coincides with the (\texttt{L}) estimator in case of the Gaussian kernel. For general kernels the analysis is highly involved; see the remark in Section~\ref{sec:asymptotics} for more details. Section~\ref{sec:LS} provides some general theory of localizing scoring rules and is of independent interest. Most proofs and some auxiliary results are contained in Appendix~\ref{app}.

\section{Local Estimation of $f$ and its Derivatives}
\label{sec:local-estimation-f}

\subsection{Preliminaries}
\label{sec:preliminaries}

By $\cC^L(\R^d)$ we denote the space of all functions $\R^d \rightarrow \R$
such that all partial derivatives of order $L$ exist and are
continuous. The set of all functions $g \in \cC^L(\R^d)$ such that all
partial derivatives of order $L$ are bounded is denoted by
$\cC^L_b(\R^d)$. For a multi-index $\balpha = (\alpha_1, \ldots, \alpha_d) \in \N_0^d$ we write $\lvert\balpha\rvert \defeq
\sum_{j=1}^d \alpha_j$ and
\begin{displaymath}
  g^{(\balpha)}(\x) \defeq \frac{\partial^{\lvert\balpha
\rvert}}{\partial x_1^{\alpha_1} \partial x_2^{\alpha_2} \cdots
    \partial x_d^{\alpha_d}} g(\x).
\end{displaymath}
The space of all Lebesgue integrable functions on $\R^d$ is denoted by $\mathcal{L}^1(\R^d)$. By $\lVert \cdot \rVert$ we denote the
Euclidean norm on $\R^d$.

For the density function $f$, we always assume that $f \in \cC_b^L(\R^d)$ for some $L \geq 2$. Note that for any function $g \in \mathcal{L}^1(\R^d) \cap  \cC^L_b(\R^d)$, we have that $g^{(\bbeta)}(\x) \to 0$ as $\lVert \x \rVert \to \infty$, whenever $|\bbeta| < L$; see Theorem \ref{thm:smooth.tails}.
In particular, $f \in \mathcal{L}^1(\R^d) \cap  \cC^L_b(\R^d)$, and we will use the property of lower order derivatives vanishing at infinity repeatedly throughout the paper.

\paragraph{The kernel function.}
All estimators involve a bounded kernel function $K: \R^d \rightarrow [0, \infty)$ such that $\int K(\z) \|\z\|^r \, d\z < \infty$ for arbitrary $r > 0$. Additional assumptions are:
\begin{itemize}[leftmargin=.75in]
\item[(K0)] $\displaystyle\int K(\z) \, dz = 1$.
\item[(K1)] $K$ is sign- and permutation-symmetric, i.e.
      \begin{displaymath} K(z_1, z_2, \ldots, z_d) = K(\xi_1
      z_{\sigma(1)}, \xi_2 z_{\sigma(2)}, \ldots, \xi_d z_{\sigma(d)})
    \end{displaymath} for all $\z \in \R^d$, $\bs{\xi} \in \{-1, 1\}^d$
    and $\sigma \in \mathcal{S}_d$, where $\mathcal{S}_d$ is the set of permutations on $\{1, 2,
    \ldots, d\}$.
\item[(K2)] $\displaystyle\int K(\z) \z \z^\top \, d\z = \I_d$,
	the identity matrix in $\R^d$.
\end{itemize}
Depending on the estimation paradigm, we may impose further conditions on
$K$. In what follows, we write $\z^\balpha \defeq \prod_{j=1}^d z_j^{\alpha_j}$for $\z \in \R^d$ and
$\balpha \in \N_0^d$. The $\balpha$-th moment of $K$ is defined as
\begin{displaymath}
	\mu_\balpha \defeq \int K(\z) \z^\balpha \, d\z .
\end{displaymath}
Note that condition (K1) implies that
\begin{equation}
\label{eq:moments.symmetry.K}
  \mu_\balpha = 0,
  \quad\text{whenever} \ \balpha \not\in 2\N_0^d,
\end{equation}
and that $\mu_\balpha$ is invariant under permutations of the components of $\balpha$. Hence, we can often rewrite expressions involving various moments $\mu_\balpha$ in terms of a few special moments, namely,
\begin{align*}
	\mu_a &\defeq \mu_{a\e_1} = \int K(\z) z_1^a \, d\z
		\quad\text{for} \ a \ge 0 ,\\
	\mu_{ab} &\defeq \mu_{a\e_1 + b\e_2} = \int K(\z) z_1^a z_2^b \, d\z
		\quad\text{for} \ a \ge b > 0 , \\
	\mu_{abc} &\defeq \mu_{a\e_1 + b\e_2 + c\e_3} = \int K(\z) z_1^a z_2^b z_3^c \, d\z
		\quad\text{for} \ a \ge b \ge c > 0 \quad (\text{if} \ d \ge 3) ,
\end{align*}
where $\e_1,\ldots,\e_d$ denotes the standard basis of $\R^d$. In particular, if (K1) is satisfied, then condition~(K2) is equivalent to $\mu_2 = 1$.

Condition~(K1) is obviously satisfied if $K$ is spherically symmetric in the sense that $K(\z) = \kappa(\|\z\|)$ for some function $\kappa : [0,\infty) \to [0,\infty)$. A particular choice for $K$ is the standard Gaussian
density $\varphi$,
\begin{displaymath}
  \varphi(\z) \defeq (2\pi)^{-d/2} \exp(-\lVert\z\rVert^2 / 2) ,
\end{displaymath}
which satisfies conditions (K0--2). We restrict our attention to scalar bandwidths rather than matrix-valued ones as in \citet{ChaconETAL2011}. The main reason, besides simplicity, is that a guiding motivation for the analysis in this paper is the estimation of ridges. Density ridges are not equivariant under arbitrary affine transformations of the coordinate system, so these structures are not necessarily preserved using matrix-valued bandwidths.

\subsection{Local Moment Matching}
\label{sec:MM}

\paragraph{Local moments.}
The concept of local moments already appears in \citet{Mueller2001} with a
different normalization and a rectangular kernel only. For our general
definition of local moments we need rescaled versions of $K$. For any
bandwidth $h > 0$ we write $K_h(\z) \defeq h^{-d} K(h^{-1} \z)$.

\begin{defi}
\label{def:local-sample-moments}
For $\balpha \in \N_0^d$ and $\x \in \R^d$, the local $\balpha$-th sample moment at $\x$ with bandwidth $h > 0$ is 
\begin{displaymath}
	s_{n,h}^\balpha(\x)
	\defeq \frac{1}{n} \sum_{i=1}^n K_h(\X_i - \x)
		(h^{-1}(\X_i - \x))^\balpha.
\end{displaymath}
\end{defi}

These local sample moments lead to estimators of the density $f(\x)$ and
partial derivatives thereof. The main ingredient is the next lemma about
the relation between the expectation of local moments and partial
derivatives of $f$. In what follows, $\bgamma \in \N_0^d$ denotes a generic multi-index, and $\bgamma! := \prod_{j=1}^d \gamma_j!$.

\begin{lem}
  \label{lem:LM-Taylor}
For a function $F: \R^d \rr \R$ with $\int \lvert F(\z) \rvert (1 +  \lVert\z\rVert^L) \, d\z < \infty$, we have 
\begin{displaymath}
  \int F(\z) f(\x + h\z) \, d\z = \sum_{\lvert\bgamma\rvert \leq L} h^{|\bgamma|}
    \frac{f^{(\bgamma)}(\x)}{\bgamma !} \int F(\z)
      \z^\bgamma \, d\z +
      \begin{cases}
        \O(h^L)& \text{uniformly in } \x, \\
        \o(h^L)& \text{as } h \rightarrow 0, \ \text{locally uniformly in
        } \x.
      \end{cases}
\end{displaymath}
In particular, for any $\balpha \in \N_0^d$,
\begin{displaymath}
  s_h^\balpha(\x) \defeq \Ex\bigl( s_{n,h}^\balpha(\x) \bigr)
  = \sum_{\lvert\bgamma\rvert \leq L} h^{\lvert\bgamma\rvert}
  \frac{f^{(\bgamma)}(\x)}{\bgamma!} \, \mu_{\balpha + \bgamma}
  +
  \begin{cases}
    \O(h^L)& \text{uniformly in } \x, \\
    \o(h^L)& \text{as } h \rightarrow 0, \ \text{locally uniformly in } \x.
  \end{cases}
\end{displaymath}
\end{lem}

\noindent
The (local) uniformity in this lemma means that $O(h^L)$ and $o(h^L)$ stand
for a remainder $h^L R(\x,h)$, where $R(\cdot, \cdot)$ is bounded, and
$\sup_{\lVert\x\rVert \le B} |R(\x,h)| \rightarrow 0$ as $h \rightarrow 0$
for any fixed $B > 0$, respectively. Note that Lemma~\ref{lem:LM-Taylor} is true for arbitrary kernels $K \ge 0$ such that $\int K(\z) (1 + \lVert\z\rVert^{\lvert\balpha\rvert + L}) \, d\z$ is finite.

\paragraph{Local moment matching.}
The basic idea of local moment matching is to determine for a given multi-index $\balpha$ with $|\balpha| \le 2$ a polynomial $p_\balpha$ of order $L$ such that
\[
	\hat{f^{(\balpha)}_{n,h}}(\x)
	\ := \ \frac{1}{nh^{|\balpha|}} \sum_{i=1}^n K_h(\X_i - \x) p_\balpha(h^{-1}(\X_i - \x))
\]
satisfies
\[
	\Ex \hat{f^{(\balpha)}_{n,h}}(\x) \ = \ f^{(\balpha)}(\x) + o(h^{L - |\balpha|}) .
\]
Indeed, it follows from Lemma~\ref{lem:LM-Taylor} that
\[
	\Ex \hat{f^{(\balpha)}_{n,h}}(\x)
	\ = \ \sum_{|\bgamma| \le L} h^{|\bgamma| - |\balpha|}
		\frac{f^{(\bgamma)}(\x)}{\bgamma!}
		\int K(\z) p_\balpha(\z) \z^\bgamma \, d\z + o(h^{L - |\balpha|}) ,
\]
and this leads to the linear equation system
\begin{equation}
\label{eq:p_alpha.L}
	\int K(\z) p_\balpha(\z) \z^\bgamma \, d\z
	\ \stackrel{!}{=} \ \one\{\bgamma = \balpha\} \balpha!
	\quad\text{whenever} \ |\bgamma| \le L .
\end{equation}
It is also instructive to interpret Lemma~\ref{lem:LM-Taylor} in terms of vector-valued first and matrix-valued second moments. In what follows, let
\begin{align*}
	\snhx \
	&\defeq \ s_{n,h}^{\bs{0}}(\x)
    	\ = \ \frac{1}{n} \sum_{i=1}^n K_h(\X_i - \x)
		\ \in \R, \\
	\bsnhx \
	&\defeq \ \bigl( s_{n,h}^{\e_j}(\x) \bigr)_{j=1}^d
		\ = \ \frac{1}{n} \sum_{i=1}^n K_h(\X_i - \x) h^{-1}(\X_i - \x)
		\ \in \R^d,\\
	\bSnhx \
	&\defeq \
		\bigl( s_{n,h}^{\e_j + \e_k}(\x) \bigr)_{j,k = 1}^d
		\ = \ \frac{1}{n} \sum_{i=1}^n K_h(\X_i - \x) h^{-2} (\X_i - \x)(\X_i - \x)^\top
		\ \in \Rddsym .
\end{align*}
Here, $\z^\top$ denotes the transpose of the vector $\z$ and $\Rddsym$
denotes the space of symmetric $(d \times d)$-matrices.
The triple $\bigl( \snhx, \bsnhx, \bSnhx \bigr)$ lies in the space
\[
	\Hds \defeq \R \times \R^d \times \Rddsym .
\]
With elementary calculations one can deduce from Lemma~\ref{lem:LM-Taylor} the following result:

\begin{cor}
\label{cor:LM-Taylor}
Under conditions (K0--2), the theoretical local moments $\shx \defeq \Ex(\snhx)$, $\bshx \defeq \Ex(\bsnhx)$ and $\bShx \defeq \Ex(\bSnhx)$ satisfy
\[
	\bigl( \shx, \bshx, \bShx \bigr)
	\ = \ \bs{J} \bigl( f(\x), hDf(\x), h^2D^2f(\x) \bigr) + o(h^2)
\]
as $h \to 0$, locally uniformly in $\x$. Here $\bs{J} : \Hds \to \Hds$ is the linear operator given by
\begin{align*}
	\bs{J} (c, \b, \A) \
	:=& \ \int K(\z) (1, \z, \z\z^\top) (c + \b^\top \z + 2^{-1} \z^\top\A\z) \, d\z \\
	=& \ \bigl( c + 2^{-1} \tr(\A), \,
		\b, \,
		c \I_d + \A \odot \M + 2^{-1} \mu_{22} \tr(\A) \I_d \bigr) ,
\end{align*}
where $\odot$ stands for componentwise multiplication, and $\M$ is the matrix with positive entries
\[
	M_{jk} \ := \ \begin{cases}
		(\mu_4 - \mu_{22})/2 & \text{if} \ j = k , \\
		\mu_{22} & \text{if} \ j \ne k .
	\end{cases}
\]
\end{cor}

The linear operator $\bs{J}$ will appear in connection with local log-likelihood methods, too. The next lemma shows that it is invertible.

\begin{lem}
\label{lem:Inverse.J}
The linear operator $\bs{J}$ in Corollary~\ref{cor:LM-Taylor} is nonsingular with inverse $\bs{J}^{-1}$ given by
\begin{equation}
\label{eq:Inverse.L}
	\bs{J}^{-1}(c,\b,\A) \
	= \ \Bigl( c - \eta^{-1} \tr(\A_o) , \,
		\b, \,
		\bigl( \A_o - (\mu_{22} - 1) \eta^{-1} \tr(\A_o) \I_d \bigr)
			\oslash \bs{M} \Bigr) ,
\end{equation}
where $\oslash$ stands for componentwise division, and
\[
	\A_o \ \defeq \ \A - c \I_d, \quad
	\eta \ \defeq \ \mu_4 + (d-1) \mu_{22} - d \ > \ 0 .
\]
\end{lem}

Thus we may estimate $\bigl( f(\x), hDf(\x), h^2D^2f(\x) \bigr)$ by
\[
	\bigl( \hat{f}_{n,h}^{\mathtt{M}}(\x),
		h\hat{Df}_{n,h}^{\mathtt{M}}(\x),
		h^2 \hat{D^2f}_{n,h}^{\mathtt{M}}(\x) \bigr)
	\ := \ \bs{J}^{-1} \bigl( \snhx, \bsnhx, \bSnhx \bigr) .
\]
Plugging in the explicit formula for $\bs{J}^{-1}$ leads to the alternative representations
\begin{align*}
	\hat{f}_{n,h}^{\mathtt{M}}(\x) \
	&= \ \frac{1}{n} \sum_{i=1}^n K_h(\X_i - \x) \,
		p_0(h^{-1}(\X_i - \x)) , \\
	\hat{Df}_{n,h}^{\mathtt{M}}(\x) \
	&= \ \frac{1}{nh} \sum_{i=1}^n K_h(\X_i - \x) \,
		\bigl( p_j(h^{-1}(\X_i - \x)) \bigr)_{j=1}^d , \\
	\hat{D^2f}_{n,h}^{\mathtt{M}}(\x) \
	&= \ \frac{1}{nh^2} \sum_{i=1}^n K_h(\X_i - \x) \,
		\bigl( p_{jk}(h^{-1}(\X_i - \x)) \bigr)_{j,k=1}^d
\end{align*}
with the polynomials
\begin{align}
\label{eq:MM2_0}
	p_0(\z) \
	&\defeq \ 1 - \eta^{-1}(\|\z\|^2 - d) , \\
\label{eq:MM2_1}
	p_j(\z) \
	&\defeq \ z_j , \\
\label{eq:MM2_2}
	p_{jk}(\z) \
	&\defeq \ \begin{cases}
		\mu_{22}^{-1} z_j z_k
			& \text{if} \ j \ne k , \\
		2 (\mu_4 - \mu_{22})^{-1}
			\bigl( z_j^2 - 1 - (\mu_{22} - 1)\eta^{-1}(\|\z\|^2 - d) \bigr)
			& \text{if} \ j = k .
		\end{cases}
\end{align}

\paragraph{A refinement.}
If we are willing to assume that $f \in \cC_b^3(\R^d)$, then even
\[
	\Ex \bigl( \hat{f}_{n,h}^{\mathtt{M}}(\x),
		h\hat{Df}_{n,h}^{\mathtt{M}}(\x),
		h^2 \hat{D^2f}_{n,h}^{\mathtt{M}}(\x) \bigr)
	\ = \ \bigl( f(\x), hDf(\x) + h^3 \bs{\beta}(\x), h^2D^2f(\x) \bigr) + o(h^3)
\]
with
\[
	\bs{\beta}(\x) \ := \ \sum_{\lvert\bgamma\rvert = 3}
		\frac{f^{(\bgamma)}(\x)}{\bgamma!} \int K(\z) \z^\bgamma \z \, d\z .
\]
This suggests to replace the polynomial $p_j(\z) := z_j$ with the polynomial $p_j(\z) = p_{\e_j}(\z)$ of order three solving \eqref{eq:p_alpha.L} with $L = 3$. For general kernels, finding $p_j$ is somewhat tedious, but in case of a spherically symmetric kernel $K$, the modified polynomial is given by
\begin{equation}
\label{eq:MM3_1}
	p_j(\z) \defeq a z_j - b \|\z\|^2 z_j
\end{equation}
with
\[
	b \defeq \Ex(R^4) \big/ \bigl( \Ex(R^6) - \Ex(R^4)^2/d \bigr)
	\quad\text{and}\quad
	a \defeq 1 + b \, \Ex(R^4)/d ,
\]
where $R := \|\bs{Z}\|$ with a random vector $\bs{Z} \sim K$; see Appendix \ref{app:MM}

\paragraph{The Gaussian kernel.}
If $K = \varphi$, the special kernel moments are $\mu_{22} = 1$ and $\mu_4 = 3$, whence $\eta = 2$. Then the estimators of $f(\x)$, $Df(\x)$ and $D^2 f(\x)$ are given by the following simplified formulae:
\begin{align*}
	\hat{f}_{n,h}^{\mathtt{M}}(\x) \
	&= \ \snhx - 2^{-1} \tr(\bSnhx - \snhx \I_d) , \\
	\hat{Df}_{n,h}^{\mathtt{M}}(\x) \
	&= \ h^{-1} \bsnhx , \\
	\hat{D^2f}_{n,h}^{\mathtt{M}}(\x) \
	&= \ h^{-2} (\bSnhx - \snhx \I_d) .
\end{align*}
The polynomials in \eqref{eq:MM2_0}, \eqref{eq:MM2_1} and \eqref{eq:MM2_2} simplify to
\[
	p_0(\z) \ = \ 1 - 2^{-1}(\|\z\|^2 - d) , \quad
	p_j(\z) \ = \ z_j
	\quad\text{and}\quad
	p_{jk}(\z) \ = \ \begin{cases}
		z_j z_k
			& \text{if} \ j \ne k , \\
		z_j^2 - 1
			& \text{if} \ j = k .
		\end{cases}
\]
Note also that
\begin{displaymath}
  \hat{f}_{n,h}^{\mathtt{M}}(\x) = n^{-1} \sum_{i=1}^n \tilde{K}_h(\X_i - \x)
\end{displaymath}
with the ``sombrero kernel'' $\tilde{K}(\z) = K(\z) (2 + d - \|\z\|^2)/2$.

As to the refined estimator of $Df(\x)$, note that $\Ex(R^4) = d(d+2)$ and $\Ex(R^6) = d(d+2)(d+4)$, and this leads to $b = 1/2$, $a = (d+4)/2$ and
\begin{align*}
	\hat{Df}_{n,h}^{\mathtt{M}}(\x) \
	&= \ \frac{1}{nh} \sum_{i=1}^n K_h(\X_i - \x)
		\, 2^{-1}(4 + d - h^{-2} \|\X_i - \x\|^2) \, h^{-1} (\X_i - \x) , \\
	p_j(\z) \
	&= \ 2^{-1}(4 + d - \|\z\|^2) z_j .
\end{align*}

\paragraph{Local Moments in \cite{Mueller2001}.}
\citet{Mueller2001} consider only the rectangular kernel $K(\z) := \
\one\{\lvert z_j\rvert \le 1 \text{ for } 1 \le j \le d\}$. The rectangular kernel
$K$ is not normalized to be a probability density and violates (K2), but it
does satisfy (K1), and $\int K(\z) \|\z\|^r \, d\z < \infty$ for all
$r \ge 0$. For $\balpha \in \N_0^d$, they propose
\[
	m_{n,h}^\balpha(\x)
	\defeq h^{-\lvert\tilde{\balpha}\rvert}
		\frac{s_{n,h}^\balpha(\x)}{s_{n,h}^{\bs{0}}(\x)}
\]
with $\tilde{\balpha} \defeq (\one\{\alpha_j \text{ is odd}\})_{j=1}^d$ as an estimator of
\[
	m^\balpha(\x) \defeq \lim_{h \to 0} h^{-\lvert\tilde{\balpha}\rvert}
		\frac{s_h^\balpha(\x)}{s_h^{\bs{0}}(\x)}
	= \frac{f^{(\tilde{\balpha})}(\x)}{f(\x)}
		\frac{\mu_{\balpha + \tilde{\balpha}}}{\mu_{\bs{0}}} .
\]
The latter equation in the previous formula is a consequence of Lemma~\ref{lem:LM-Taylor} and \eqref{eq:moments.symmetry.K}, because
\begin{align*}
	s_h^{\bs{0}}(\x)
	&= \mu_{\bs{0}} f(\x) + O(h^2), \\
	s_h^{\balpha}(\x)
	&= \sum_{\lvert\bgamma\rvert \le \lvert\tilde{\balpha}\rvert}
		\frac{h^{\lvert\bgamma\rvert}}{\bgamma!}
			f^{(\bgamma)}(\x) \mu_{\balpha+\bgamma}
		+ o(h^{\lvert\tilde{\balpha}\rvert})
		= h^{\lvert\tilde{\balpha}\rvert} f^{(\tilde{\balpha})}(\x)
			\mu_{\balpha+\tilde{\balpha}} + o(h^{\lvert\tilde{\balpha}\rvert})
\end{align*}
as $h \to 0$, locally uniformly in $\x$, whenever
$\lvert \tilde{\balpha} \rvert \leq L$. The last equality in the previous
equation follows because $\balpha + \bgamma \in 2\N_0^d$ can only hold if
$\lvert \bgamma \rvert \geq \lvert \tilde{\balpha} \rvert$. 

The problem with the limit $m^\balpha(\x)$ is that sometimes the most
interesting part of $s_{n,h}^\balpha(\x)$ is hidden in the higher order
terms of its Taylor expansion, whereas $m^\balpha(\x)$ is
non-informative. For instance, $m^{2\e_j}(\x) = \mu_2/\mu_0$ for $1 \le j
\le d$, which does not even depend on $f$. Hence, we are not using the quantities $m^\balpha(\x)$.

\subsection{Kernel Density Estimation}
\label{sec:KD}

The local $\bs{0}$-th sample moment $\snhx = s_{n,h}^{\bs{0}}(\x)$ is just the usual kernel density estimator with kernel $K$ and bandwidth $h$ at the point $\x$, as proposed by \cite{Rosenblatt1956}, \cite{Parzen1962} for $d = 1$, see also \citet{WandJones1995}. Hence we write it as $\hat{f}_{n,h}^{}(\x)= \hat{f}_{n,h}^{\mathtt{K}}(\x)$. Assuming that $K \in \cC_b^2(\R^d)$, the kernel density estimator $\hat{f}_{n,h}$ belongs to $\cC_b^2(\R^d)$ as well, and a natural estimator for $f^{(\balpha)}$ with $\balpha \in \N_0^d$, $\lvert\balpha\rvert \le 2$, is given by $\hat{f_{n,h}^{(\balpha)}} \defeq (\hat{f}_{n,h}^{})^{(\balpha)}$, i.e.
\begin{displaymath}
  \hat{f_{n,h}^{(\balpha)}}(\x)
  = \frac{(-1)^{\lvert\balpha\rvert}}{nh^{d+ \lvert\balpha\rvert}}
	\sum_{i=1}^n K^{(\balpha)}(h^{-1}(\X_i - \x))
  = \frac{(-1)^{\lvert\balpha\rvert}}{nh^{\lvert\balpha\rvert}}
  	\sum_{i=1}^n (K^{(\balpha)})_h (\X_i - \x) ,
\end{displaymath}
where $G_h(\z) := h^{-d} G(h^{-1} \z)$ for any real-, vector- or matrix-valued function $G$ on $\R^d$. In particular, $f(\x)$, $Df(\x)$ and $D^2f(\x)$ are estimated by
\begin{align*}
  \hat{f}_{n,h}^{\mathtt{K}}(\x)
  &= \frac{1}{n} \sum_{i=1}^n K_h(\X_i - \x),\\
  \hat{D f}_{n,h}^{\mathtt{K}}(\x)
  &= - \frac{1}{nh} \sum_{i=1}^n (D K)_h(\X_i - \x),\\
  \hat{D^2f}_{n,h}^{\mathtt{K}}(\x)
  &= \frac{1}{nh^2} \sum_{i=1}^n (D^2K)_h(\X_i - \x).
\end{align*}

\paragraph{The Gaussian kernel.}
In the special case of $K = \varphi$, gradient and Hessian matrix of $K$
are given by $DK(\z) = - K(\z) \z$ and
$D^2K(\z) = K(\z) (\z \z^\top - \I_d)$, respectively. The kernel density
estimators are then
\begin{align*}
  \hat{f}_{n,h}^{\mathtt{K}}(\x)
  & = \snhx,\\
  \hat{Df}_{n,h}^{\mathtt{K}}(\x)
  & = h^{-1} \bsnhx,\\
  \hat{D^2f}_{n,h}^{\mathtt{K}}(\x)
  & = h^{-2} \bigl( \bSnhx - \snhx \I_d \bigr) .
\end{align*}

\subsection{Local Log-Likelihood Estimation}
\label{sec:LL}

From now on we focus on the log-density $\ell \ := \ \log f$
and the region $\{f > 0\} = \{\ell > -\infty\}$. For $\x,\y \in \{f > 0\}$,
\begin{displaymath}
	\ell(\y) \ = \ \ell(\x)
		+ D\ell(\x)^\top (\y-\x)
		+ 2^{-1} (\y-\x)^\top D^2\ell(\x) (\y-\x)
		+ \o(\lVert\y - \x\rVert^2)
\end{displaymath}
as $\y \rightarrow \x$. Thus, in a small neighborhood of $\x$, $f$ may be
approximated by $f_{\ell(\x),D\ell(\x),D^2\ell(\x)}(\cdot - \x)$, where
\[
	f_{c,\b,\A}(\z) \defeq \exp\bigl( g_{c,\b,\A}(\z) \bigr)
	\quad\text{with}\quad
	g_{c,\b,\A}(\z) \defeq c + \z^\top\b + 2^{-1} \z^\top \A \z
\]
with a parameter $(c,\b,\A) \in \Hds = \R \times \R^d \times \Rddsym$. \citet{Loader1996} proposed to estimate the unknown triple $\bigl( \ell(\x), D\ell(\x), D^2\ell(\x) \bigr)$ by the minimizer $\bigl( \hat{\ell}_{n,h}^{\mathtt{L}}(\x), \hat{D\ell}_{n,h}^{\mathtt{L}}(\x), \hat{D^2\ell}_{n,h}^{\mathtt{L}}(\x) \bigr)$ of the local negative log-likelihood function
\begin{align}
\label{eq:LL-Functional}
	\hat{S}_{n,h}^{\mathtt{L}}(c,\b,\A,\x) \
	\defeq& \ - \frac{1}{n} \sum_{i=1}^n K_h(\X_i - \x) g_{c,\b,\A}(\X_i - \x)
		+ \int K_h(\y - \x) f_{c,\b,\A}(\y - \x) \, d\y \\
\nonumber
	=& \ - \frac{1}{n} \sum_{i=1}^n K_h(\X_i - \x) g_{c,\b,\A}(\X_i - \x)
		+ \int K(\z) f_{c,h\b,h^2\A}(\z) \, d\z
\end{align}
over $(c,\b,\A) \in \Hds$. In general, a minimizer may not exist. But under an exponential moment condition on $K$, $\hat{S}_{n,h}^{\mathtt{L}}(\cdot,\cdot,\cdot,\x)$ has a unique minimizer with asymptotic probability one as $h \to 0$ and $nh^d \to \infty$; see Section~\ref{app:LL}. If a minimizer exists, it fulfills the following equation system:
\[
	\bigl( \snhx, \bsnhx, \bSnhx \bigr)
	\ = \ \int K(\z) (1, \z, \z\z^\top) f_{c,h\b,h^2\A}(\z) \, d\z ,
\]
where on the left hand side we have all the empirical local moments of order at most two. For fixed $(c,\b,\A)$ and as $h \to 0$, the right hand side is approximately equal to
\[
	\int K(\z) (1, \z, \z\z^\top) \exp(c)
		\bigl( 1 + h\b^\top\z + 2^{-1} h^2 \z^\top (\A + \b\b^\top) \z \bigr) \, d\z
	\ = \ \exp(c) \bs{J} \bigl( 1, h\b, h^2(\A + \b\b^\top) \bigr)
\]
with the operator $\bs{J}$ in Corollary~\ref{cor:LM-Taylor}. This indicates that local moment matching and local log-likelihood are closely related, at least asymptotically, and this will be made precise in Section~\ref{sec:asymptotics}.

\paragraph{The Gaussian kernel.}

For general $K$, a closed form expression for $\bigl( \hat{\ell}_{n,h}^{\mathtt{L}}(\x), \hat{D\ell}_{n,h}^{\mathtt{L}}(\x), \hat{D^2\ell}_{n,h}^{\mathtt{L}}(\x) \bigr)$ does not exist. But in case of $K = \varphi$, we have
\begin{align*}
	\hat{S}_{n,h}^{\mathtt{L}}(c, \b, \A, \x)
	= &- \ c s_n(\x)
		- h \b^\top \bsnhx
		- 2^{-1} h^2\tr(\A\bSnhx)\\
	&+ \ \exp \Bigl( c + 2^{-1} h^2 \b^\top (\I_d - h^2 \A)^{-1} \b \Bigr)
		\det(\I_d - h^2 \A)^{-1/2},
\end{align*}
provided that $\I_d - h^2\A$ is positive definite. Otherwise
$\hat{S}_{n,h}^{\mathtt{L}}(c, \b, \A, \x) = \infty$. As shown in Appendix
\ref{app:LL}, this leads to the estimators
\begin{align*}
	\hat{\ell}_{n,h}^{\mathtt{L}}(\x)
	&= \log s_n(\x)
		- 2^{-1} \munhx^\top \Sigmanhx \munhx
		- 2^{-1} \log \det \Sigmanhx , \\
	\hat{D\ell}_{n,h}^{\mathtt{L}}(\x)
	&= h^{-1} \Sigmanhx^{-1} \munhx , \\
	\hat{D^2 \ell}_{n,h}^{\mathtt{L}}(\x)
	&= h^{-2} \bigl( \I_d - \Sigmanhx^{-1} \bigr)
\end{align*}
with the local sample mean
\begin{equation}\label{eq:local-mean}
	\munhx \ \defeq \ \snhx^{-1} \bsnhx
	\ = \ \int \y \, \hat{P}_{n,h}(d\y)
\end{equation}
and the local sample covariance matrix
\begin{equation} \label{eq:local-variance}
  \Sigmanhx \ \defeq \ \snhx^{-1} \bSnhx - \munhx \munhx^\top
  \ = \ \int (\y - \munhx)(\y - \munhx)^\top \,
  \hat{P}_{n,h}(d\y)
\end{equation}
of the local empirical distribution
\[
	\hat{P}_{n,h} \ := \ \Bigl( \sum_{i=1}^n K_h(\X_i - \x) \Bigr)^{-1}
		\sum_{i=1}^n K_h(\X_i - \x) \delta_{h^{-1}(\X_i - \x)}^{} .
\]
If $n \ge d+1$, then with probability one, the matrix $\Sigmanhx \in \Rddsym$ is nonsingular for each $\x \in \R^d$. This follows from the fact that with probability one, the points $\X_1, \X_2, \ldots, \X_n$ are not contained in a hyperplane in $\R^d$.

\subsection{Local Hyv\"arinen Score Estimation}
\label{sec:HS}

Let $\PP \subset \CC^2(\R^d)$ be a family of integrable, strictly positive functions. Suppose further that for arbitrary $p, q \in \PP$,
\[
	\int p(\y) \bigl( \lVert D \log q(\y)\rVert^2
		+ \lVert D^2 q(\y)\rVert_F^2/q(\y) \bigr) \, d\y
	\ < \ \infty
\]
and
\[
  p(\y) \lVert D \log q(\y)\rVert \ \to \ 0 \quad\text{as} \ \|\y\| \to
  \infty ,
\]
where $\lVert \cdot \rVert_F$ is the Frobenius norm given by
$\lVert \A \rVert_F^2 = \tr(\A\A^\top)$. An example is the family of
multivariate Gaussian distributions. \citet{Hyvaerinen2005} proposed the
scoring function $S^{\mathtt{H}} : \PP \times \R^d \to \R$ given by
\[
	S^{\mathtt{H}}(p, \y)
	\ \defeq \ 2^{-1} \lVert D \log p(\y)\rVert^2 + \triangle \log p(\y) ,
\]
where $\triangle$ is the Laplace operator, i.e.\ $\triangle g(\y) \defeq \tr(D^2 g(\y))$. With $S^{\mathtt{H}}(q, p) \defeq \int S^{\mathtt{H}}(q,\y) p(\y) \, d\y$ for $p,q \in \PP$, he showed that
\[
	S^{\mathtt{H}}(q,p) \ = \ S^{\mathtt{H}}(p,p)
		+ 2^{-1} \int \bigl\lVert D \log q(\y) - D \log p(\y) \bigr\rVert^2 p(\y) \, d\y .
\]
That means, $S^{\mathtt{H}}(q,p) \ge S^{\mathtt{H}}(p,p)$ with equality if and only if $q = cp$ for some $c > 0$.

In our setting of density estimation, we assume that for given $\x \in \{f > 0\}$ and sufficiently small bandwidth $h > 0$, the product $K_h(\cdot - \x) f(\cdot)$ may be approximated by $f(\x) p_{D\ell(\x),D^2\ell(\x)}^{}(\cdot - \x)$, where
\[
	p_{\b,\A}^{}(\y)
	\ \defeq \ K_h(\y) \exp \bigl( g_{0,\b,\A}(\y) \bigr)
	\ = \ K_h(\y) \exp(\b^\top\y + 2^{-1} \y^\top\A\y)
\]
for $\b \in \R^d$ and $\A \in \Rddsym$. Since the empirical measure
\[
	\frac{1}{n} \sum_{i=1}^n K_h(\X_i - \x) \delta_{\X_i}^{}
\]
may be viewed as an estimator of the measure with density $K_h(\cdot - \x) f(\cdot) \approx f(\x) p_{D\ell(\x),D^2\ell(\x)}^{}(\cdot - \x)$, one can try to estimate $\bigl( D\ell(\x), D^2\ell(\x) \bigr)$ by a minimizer of
\[
	\frac{1}{n} \sum_{i=1}^n K_h(\X_i - \x) S^{\mathtt{H}}(p_{\b,\A}^{}, \X_i - \x)
\]
over all $(\b,\A) \in \Hds_0$, where
\[
	\Hds_0 \ := \ \R^d \times \Rddsym .
\]
We assume temporarily that the family $\PP := \bigl\{ p_{\b,\A}^{} : (\b,\A) \in \Hds_0 \bigr\}$ satisfies the assumptions above. But note that
\[
	D \log p_{\b,\A}^{}(\y) \ = \ D \log K_h(\y) + \b + \A\y
	\quad\text{and}\quad
	\triangle \log p_{\b,\A}^{}(\y) \ = \ \triangle \log K_h(\y) + \tr(\A) ,
\]
so
\[
	K_h(\y) S^{\mathtt{H}}(p_{\b,\A}^{},\y)
	\ = \ 2^{-1} K_h(\y) \|\b + \A\y\|^2
		+ K_h(\y) \tr(\A) + h^{-1} (DK)_h(\y)^\top (\b + \A \y) + \kappa_h(\y)
\]
for some $\kappa_h(\y)$ not depending on $(\b,\A)$. Hence, it suffices to assume that $K \in \cC^1(\R^d)$, and we propose to estimate $\bigl( D\ell(\x), D^2\ell(\x) \bigr)$ by a minimizer $\bigl( \hat{D\ell}_{n,h}^{\mathtt{H}}(\x), \hat{D^2\ell}_{n,h}^{\mathtt{H}}(\x) \bigr)$ of
\[
	\hat{S}_{n,h}^{\mathtt{H}}(\b,\A,\x) \
	\defeq \ \frac{1}{n}\sum_{i=1}^n S_h^{\mathtt{H}}(\b,\A, \X_i - \x)
\]
over all $(\b,\A) \in \Hds_0$, where
\begin{displaymath}
	S_h^{\mathtt{H}}(\b,\A,\y) \
	\defeq \ 2^{-1} K_h(\y) \|\b + \A\y\|^2
		+ K_h(\y) \tr(\A) + h^{-1} (DK)_h(\y)^\top (\b + \A \y) .
\end{displaymath}
As shown in Section~\ref{app:LH}, the resulting estimator is given by
\begin{align}
\label{eq:HS_1}
	h \hat{D\ell}_{n,h}^{\mathtt{H}}(\x) \
	&= \ - h^2 \hat{D^2\ell}_{n,h}^{\mathtt{H}}(\x) \munhx -  \bqnhx , \\
\label{eq:HS_2}
	h^2 \hat{D^2\ell}_{n,h}^{\mathtt{H}}(\x) \
	&= \ \bs{T}_{\Sigmanhx}^{} \bigl( - \I_d - \bQnhx + \munhx \bqnhx^\top \bigr) ,
\end{align}
where $\munhx$ are defined at \eqref{eq:local-mean}, $\Sigmanhx$ at
\eqref{eq:local-variance}, and
\begin{align*}
	\bqnhx \
	&\defeq \ \Bigl( \sum_{i=1}^n K_h(\X_i - \x) \Bigr)^{-1}
		\sum_{i=1}^n (DK)_h(\X_i - \x)
		\ \in \R^d , \\
	\bQnhx \
	&\defeq \ \Bigl( \sum_{i=1}^n K_h(\X_i - \x) \Bigr)^{-1}
		\sum_{i=1}^n (DK)_h(\X_i - \x) h^{-1} (\X_i - \x)^\top
		\ \in \R^{d\times d} .
\end{align*}
Here, $\bs{T}_{\bSigma}(\B)$ denotes the unique solution $\A \in \Rddsym$ of the equation
\[
	\bSigma\A + \A \bSigma \ = \ \B + \B^\top
\]
for positive definite $\bSigma \in \Rddsym$ and $\B \in \R^{d\times d}$. An explicit representation of $\bs{T}_{\bSigma}(\B)$ is provided in Section~\ref{app:LH}. If $\B$ is symmetric and $\bSigma \B = \B\bSigma$, then $\T_{\bSigma}(\B) = \bSigma^{-1}\B$.

\paragraph{The Gaussian kernel.}
In case of $K = \varphi$,
\[
	\bqnhx \ = \ - \munhx
	\quad\text{and}\quad
	\bQnhx \ = \ - \Sigmanhx + \munhx \munhx^\top ,
\]
so $-\I_d - \bQnhx +  \munhx \bqnhx^\top = \Sigmanhx - \I_d$ and $\bs{T}_{\Sigmanhx}^{} \bigl( \Sigmanhx - \I_d \bigr) = \bs{I} - \Sigmanhx^{-1}$. Consequently, \eqref{eq:HS_1} and \eqref{eq:HS_2} yield the explicit formulae
\begin{align*}
	\hat{D\ell}_{n,h}^{\mathtt{H}}(\x) \
	&= \ h^{-1} \Sigmanhx^{-1} \munhx
		\ = \ \hat{D\ell}_{n,h}^{\mathtt{L}}(\x) , \\
	\hat{D^2 \ell}_{n,h}^{\mathtt{H}}(\x) \
	&= \ h^{-2} \bigl( \I_d - \Sigmanhx^{-1} \bigr)
		\ = \ \hat{D^2 \ell}_{n,h}^{\mathtt{L}}(\x) ,
\end{align*}
which show that the local log-likelihood estimator coincides with the local
Hyv\"arinen score estimator for the Gaussian kernel.

\section{Asymptotic Properties}
\label{sec:asymptotics}

In this section we analyze the asymptotic properties of the estimation paradigms (\texttt{M}), (\texttt{K}) and (\texttt{L}) under conditions (K0--2), assuming that $f \in \cC_b^4(\R^d)$. All asymptotic statements refer to the following setting:
\[
	\x \to \x_o \in \{f > 0\}, \quad
	h \to 0
	\quad\text{and}\quad
	nh^d \to \infty .
\]
For notational convenience, some of our results are formulated in terms of directional derivatives
\[
	D^k g(\x; \z)
	\defeq \left. \frac{d^k}{dt^k} \right\vert_{t=0} g(\x + t\z)
	= k! \sum_{|\bgamma| = k} \frac{g^{(\bgamma)}(\x)}{\bgamma!} \, \z_{}^\bgamma
\]
of $g = f$ or $g = \ell$.

Note that $\Hds$ is a Euclidean space with inner product
\[
	\bigl\langle (c,\b,\A), (\tilde{c},\tilde{\b},\tilde{\A}) \bigr\rangle
	\ \defeq \ c\tilde{c} + \b^\top \tilde{\b} + 2^{-1} \tr(\A \tilde{\A})
\]
and norm $\lVert (c,\b,\A)\rVert = \bigl( c^2 + \|\b\|^2 + 2^{-1} \|\A\|_F^2 \bigr)^{1/2}$. The factor $2^{-1}$ in front of $\tr(\A \tilde{\A})$ looks a bit arbitrary, but it will be convenient in connection with local log-likelihood.

\subsection{Linear Expansions}

The estimators for $g = f$ or $g = \ell$ and its derivatives admit the following linear expansion:

\begin{equation}
  \tag{LE}
  \Bigl( \hat{g}_{n,h}(\x), h \hat{Dg}_{n,h}(\x), h^2 \hat{D^2g}_{n,h}(\x) \Bigr)
  \ = \ \bigl( g(\x), h Dg(\x), h^2 D^2g(\x) \bigr)
  + \Bvechx + \Yvecnhx
  \label{eq:LE}
\end{equation}
with a deterministic bias $\Bvechx$ and a stochastic term $\Yvecnhx$. Precisely,
\[
	\Bvechx
	\ = \ 
	\Bigl( h^{\gamma(0)} \beta(\x_o) + o(h^{\gamma(0)}), \
			h^{\gamma(1)} \bs{\beta}(\x_o) + o(h^{\gamma(1)}), \
			h^4 \B(\x_o) + o(h^4) \Bigr)
\]
with given exponents $\gamma(0) \in \{2,3,4\}, \gamma(1) \in \{3,4\}$ and a triple $\bigl( \beta(\x_o), \bs{\beta}(\x_o), \B(\x_o) \bigr) \in \Hds$, and

\begin{equation}
  \label{eq:stochastic-term}
  \Yvecnhx \ = \ \frac{1}{nh^d} \sum_{i=1}^n
  \bigl[ \Gvec(h^{-1}(\X_i - \x)) - \Ex \Gvec(h^{-1}(\X_1 - \x)) \bigr]
  + o_p((nh^d)^{-1/2})                        
\end{equation}
with a certain function $\Gvec : \R^d \to \Hds$ such that
\[
	\int \bigl( \lVert\Gvec(\z)\rVert + \lVert\Gvec(\z)\rVert^2 \bigr) \, d\z
	\ < \ \infty .
\]
The function $\Gvec$ depends on derivatives of the kernel possibly multiplied with a multivariate polynomial in $\z$.

\begin{rem}
We do not analyse the asymptotic properties of the estimator resulting from paradigm (\texttt{H}). The reasons for this are that under the Gaussian kernel, the estimators (\texttt{L}) and (\texttt{H}) coincide, and while the stochastic term for (\texttt{H}) is easy to handle, the bias term is excessively technical for general kernels.
\end{rem}

\paragraph{Asymptotic normality.}
The representation of the stochastic term at \eqref{eq:stochastic-term}
implies that
\begin{equation}
\label{eq:as.normality}
	(nh^d)^{1/2} \Yvecnhx \ \rd \ \cN(0, f(\x_o) \bSigma)
\end{equation}
in the sense that for any $\Hvec \in \Hds$,
\[
	\bigl\langle (nh^d)^{1/2} \Yvecnhx, \Hvec \bigr\rangle
	\ \rd \ \cN \bigl( 0, f(\x_o) \bSigma(\Hvec) \bigr)
	\quad\text{with}\quad
	\bSigma(\Hvec) \ := \ \int \langle \Gvec(\z), \Hvec\rangle^2 \, d\z .
\]
The detailed arguments are given in Appendix \ref{app:LE}.

\paragraph{From densities to log-densities.}
The paradigms (\texttt{M}) and (\texttt{K}) yield estimators $\hat{f}_{n,h}^{}(\x)$, $\hat{Df}_{n,h}^{}(\x)$ and $\hat{D^2f}_{n,h}^{}(\x)$ of $f(\x)$, $Df(\x)$ and $D^2f(\x)$. Concerning $\ell = \log f$, note that
\[
	D\ell(\x) \ = \ f(\x)^{-1} Df(\x)
	\quad\text{and}\quad
	D^2\ell(\x) \ = \ f(\x)^{-1} D^2f(\x) - f(\x)^{-2} Df(\x) Df(\x)^\top .
\]
Hence, natural estimators for $\ell(\x)$, $D\ell(\x)$ and $D^2\ell(\x)$ are given by
\begin{align}
\label{eq:hat.ell.0}
	\hat{\ell}_{n,h}^{}(\x) \
	&\defeq \ \log \hat{f}_{n,h}^{}(\x) , \\
\label{eq:hat.ell.1}
	\hat{D\ell}_{n,h}^{}(\x) \
	&\defeq \ \hat{f}_{n,h}^{}(\x)^{-1} \hat{Df}_{n,h}^{}(\x), \\
\label{eq:hat.ell.2}
	\hat{D^2\ell}_{n,h}^{}(\x) \
	&\defeq \ \hat{f}_{n,h}^{}(\x)^{-1} \hat{D^2f}_{n,h}^{}(\x)
		- \hat{f}_{n,h}^{}(\x)^{-2}
			\hat{Df}_{n,h}^{}(\x) \hat{Df}_{n,h}^{}(\x)^\top .
\end{align}
Interestingly, the asymptotic properties of these estimators are similar to the asymptotic properties of $\hat{f}_{n,h}^{}(\x)$, $\hat{Df}_{n,h}^{}(\x)$ and $\hat{D^2f}_{n,h}^{}(\x)$.

\begin{thm}[Asymptotics for $\hat{\ell}$ in terms of $\hat{f}$]
\label{thm:from.f.to.ell}
Suppose that the estimators $\hat{f}_{n,h}^{}(\x)$, $\hat{Df}_{n,h}^{}(\x)$ and $\hat{D^2f}_{n,h}^{}(\x)$ admit a linear expansion (LE). Then the estimators $\hat{\ell}_{n,h}^{}(\x)$ $\hat{D\ell}_{n,h}^{}(\x)$ and $\hat{D^2\ell}_{n,h}^{}(\x)$ given by \eqref{eq:hat.ell.0}, \eqref{eq:hat.ell.1} and \eqref{eq:hat.ell.2} admit a linear expansion (LE) with the following new ingredients:
\begin{align*}
	\Gvec^{\mathtt{new}}(\z) \
	&= \ f(\x_o)^{-1} \Gvec(\z) , \\
	\gamma^{\mathtt{new}}(0) \
	&= \ \gamma(0) , \\
	\beta^{\mathtt{new}}(\x_o) \
	&= \ f(\x_o)^{-1} \beta(\x_o) , \\
	\gamma^{\mathtt{new}}(1) \
	&= \ \min\{\gamma(1),\gamma(0)+1\} , \\
	\bs{\beta}^{\mathtt{new}}(\x_o) \
	&\defeq \ f(\x_o)^{-1} \Bigl(
		\one\{\gamma(1) = \gamma^{\mathtt{new}}(1)\} \bs{\beta}(\x_o)
		- \one\{\gamma(0) + 1 = \gamma^{\mathtt{new}}(1)\}
			f(\x_o)^{-1} \beta(\x_o) D\ell(\x_o) \Bigr) , \\
	\B^{\mathtt{new}}(\x_o) \
	&\defeq \ f(\x_o)^{-1} \Bigl( \B(\x_o)
		- \one\{\gamma(0) = 2\} \beta(\x_o)
			\bigl( D^2\ell(\x_o) - D\ell(\x_o) D\ell(\x_o)^\top \bigr) \\
	& \qquad - \ \one\{\gamma(1) = 3\}
			\bigl( \bs{\beta}(\x_o) D\ell(\x_o)^\top
					+ D\ell(\x_o)\bs{\beta}(\x_o)^\top \bigr) \Bigr) .
\end{align*}
\end{thm}

\begin{rem}
The bias term $\Bvec_h^{\mathtt{new}}(\x)$ looks rather complicated. But in
case of $\gamma(0) = \gamma(1) = 4$, which is the case for the refined
local moment matching estimator, we have the simplified identities
\begin{align*}
	\Gvec^{\mathtt{new}}(\z) \
	&= \ f(\x_o)^{-1} \Gvec(\z) , \\
	\gamma^{\mathtt{new}}(j) \
	&= \ 4 \quad\text{for} \ j = 0,1, \\
	\beta^{\mathtt{new}}(\x_o) \
	&= \ f(\x_o)^{-1} \beta(\x_o) , \\
	\bs{\beta}^{\mathtt{new}}(\x_o) \
	&\defeq \ f(\x_o)^{-1} \bs{\beta}(\x_o) , \\
	\B^{\mathtt{new}}(\x_o) \
	&\defeq \ f(\x_o)^{-1} \B(\x_o) .
\end{align*}
\end{rem}

\subsection{Local Moment Matching}
\label{sec:LM-asym}

For local moment matching, we need no further assumptions on the kernel $K$.

\begin{thm}[Asymptotics for (\texttt{M})]
\label{thm:MM_f}
The estimators $\hat{f}_{n,h}^{\mathtt{M}}(\x)$, $\hat{Df}_{n,h}^{\mathtt{M}}(\x)$ and $\hat{D^2f}_{n,h}^{\mathtt{M}}(\x)$ admit a linear expansion (LE) with the following ingredients:
\[
	\Gvec(\z) \ = \ K(\z) \Bigl(
		p_0(\z), \,
		\bigl( p_j(\z) \bigr)_{j=1}^d, \,
		\bigl( p_{jk}(\z) \bigr)_{j,k=1}^d \Bigr)
\]
with the polynomials $p_0$ given by \eqref{eq:MM2_0}, $p_{jk}$ given by \eqref{eq:MM2_2} and $p_j$ given by \eqref{eq:MM2_1} or \eqref{eq:MM3_1}. Moreover,
\begin{align*}
	\gamma(0) = 4 \quad\text{and}\quad
	\beta(\x_o)
	&= 24^{-1} \int K(\z) p_0(\z) D^4f(\x_o;\z) \, d\z , \\
	\B(\x_o)
	&= 24^{-1} \int K(\z) \bigl( p_{jk}(\x) \bigr)_{j,k=1}^d D^4f(\x_o;\z) \, d\z .
\end{align*}
In case of $p_j(\z) = z_j$,
\[
	\gamma(1) = 3
	\quad\text{and}\quad
	\bs{\beta}(\x_o)
		= 6^{-1} \int K(\z) D^3f(\x_o; \z) \, \z \, d\z ,
\]
In case of a spherically symmetric kernel $K$ and $p_j(\z) = a z_j - b \|\z\|^2 z_j$ as in \eqref{eq:MM3_1},
\[
	\gamma(1) = 4
	\quad\text{and}\quad
	\bs{\beta}(\x_o) = \bs{0} .
\]
\end{thm}

\begin{rem}
\label{rem:MM_f}
In the case of local moment matching for the estimation of $f$, $Df$ and $D^2f$, the stochastic term $\Yvecnhx$ is equal to the sum $(nh^d)^{-1} \sum_{i=1}^n \bigl[ \Gvec(h^{-1}(\X_i - \x)) - \Ex \Gvec(h^{-1}(\X_1 - \x)) \bigr]$ without an additional remainder $o_p \bigl( (nh^d)^{-1/2} \bigr)$. In the proof of Theorem~\ref{thm:MM_f} it is also shown that
\[
	\Ex \bigl( \|\Yvecnhx\|^2 \bigr) \ = \ O \bigl( (nh^d)^{-1} \bigr)
	\quad\text{and}\quad
	\Bvechx \ = \ \bigl( O(h^4), O(h^{\gamma(1)}), O(h^4) \bigr)
\]
uniformly in $\x \in \R^d$.
\end{rem}

\begin{rem}
Consider $\hat{f}_{n,h}^{\mathtt{M}}$, $\hat{Df}_{n,h}^{\mathtt{M}}$ and $\hat{D^2f}_{n,h}^{\mathtt{M}}$ with a spherically symmetric kernel $K$ and the refined polynomials $p_j$ in \eqref{eq:MM3_1}. Then $\gamma(0) = \gamma(1) = 4$, and the remark after Theorem~\ref{thm:from.f.to.ell} about the simplified asymptotics for $\hat{\ell}_{n,h}^{\mathtt{M}}$, $\hat{D\ell}_{n,h}^{\mathtt{M}}$ and $\hat{D^2\ell}_{n,h}^{\mathtt{M}}$ applies.
\end{rem}

\paragraph{Rates of convergence.} Theorems~\ref{thm:from.f.to.ell} and \ref{thm:MM_f} imply the following rates of convergence for estimators of $g = f, \ell$ and its derivatives. If $p_j(\z) = z_j$,
\begin{align*}
	\hat{g}_{n,h}^{\mathtt{M}}(\x) \
	&= \ g(\x) + O_p(n^{-1/2} h^{-d/2}) + O(h^4) , \\
	\hat{Dg}_{n,h}^{\mathtt{M}}(\x) \
	&= \ Dg(\x) + O_p(n^{-1/2}h^{-d/2-1}) + O(h^2), \\
	\hat{D^2g}_{n,h}^{\mathtt{M}}(\x) \
	&= \ D^2g(\x) + O_p(n^{-1/2} h^{-d/2-2}) + O(h^2) .
\end{align*}
In case of
\[
	h \ = \ C n^{-1/(8 + d)}
\]
for some $C > 0$, we obtain
\begin{align*}
	\hat{g}_{n,h}^{\mathtt{M}}(\x) \
	&= \ g(\x) + O_p(n^{-4/(8+d)}) , \\
	\hat{Dg}_{n,h}^{\mathtt{M}}(\x) \
	&= \ Dg(\x) + O_p(n^{-2/(8+d)}) , \\
	\hat{D^2g}_{n,h}^{\mathtt{M}}(\x) \
	&= \ D^2g(\x) + O_p(n^{-2/(8+d)}) .
\end{align*}
For the first derivative, one could choose a smaller bandwidth
\[
	h \ = \ C n^{-1/(6+d)}
\]
and then obtain
\[
	\hat{Dg}_{n,h}^{\mathtt{M}}(\x)
	\ = \ Dg(\x) + O_p(n^{-2/(6+d)}) .
\]
But if $K$ is spherically symmetric and $p_j(\z) = az_j - b \|\z\|^2 z_j$ as in \eqref{eq:MM3_1}, then
\begin{align*}
	\hat{Dg}_{n,h}^{\mathtt{M}}(\x) \
	&= \ Dg(\x) + O_p(n^{-1/2}h^{-d/2-1}) + o(h^3) \\
	&= \ Dg(\x) + O_p(n^{-3/(8+d)}) \quad\text{if} \ h = C n^{-1/(8+d)} .
\end{align*}

\subsection{Kernel Density Estimation}
\label{sec:DKE-asym}

For the kernel density estimators we make an additional assumption on the kernel $K$.
\begin{itemize}[leftmargin=.75in]
\item[(K3$^{\mathtt{K}}$)] \ $K \in \mathcal{C}_b^2(\R^d)$,
	and $K^{(\e_m)} \in \LL^1(\R^d)$ for $1 \le m \le d$.
\end{itemize}
Then one can easily verify the following theorem.

\begin{thm}[Asymptotics for (\texttt{K})]
\label{thm:KD_f}
The estimators $\hat{f}_{n,h}^{\mathtt{K}}(\x)$, $\hat{Df}_{n,h}^{\mathtt{K}}(\x)$ and $\hat{D^2f}_{n,h}^{\mathtt{K}}(\x)$ admit a linear expansion (LE) with the following ingredients:
\[
	\Gvec(\z)
	\ = \ \bigl( K(\z), DK(\z), D^2K(\z) \bigr) ,
\]
while
\begin{align*}
	\gamma(0) = 2 \quad\text{and}\quad
	\beta(\x_o) \
	&= \ 2^{-1} \sum_{m=1}^d f^{(2\e_m)}(\x_o) , \\
	\gamma(1) = 3 \quad\text{and}\quad
	\bs{\beta}(\x_o) \
	&= \ 2^{-1} \Bigl( \sum_{m=1}^d f^{(\e_j + 2\e_m)}(\x_o)\Bigr)_{j=1}^d , \\
	\B(\x_o) \
	&= \ 2^{-1} \Bigl( \sum_{m=1}^d f^{(\e_j+\e_k + 2\e_m)}(\x_o)\Bigr)_{j,k=1}^d .
\end{align*}
\end{thm}

\begin{rem}
Also in the case of kernel density estimation, the stochastic term $\Yvecnhx$ is equal to the sum $(nh^d)^{-1} \sum_{i=1}^n \bigl[ \Gvec(h^{-1}(\X_i - \x)) - \Ex \Gvec(h^{-1}(\X_1 - \x)) \bigr]$ without an additional remainder $o_p \bigl( (nh^d)^{-1/2} \bigr)$. As in the proof of Theorem~\ref{thm:MM_f} one can show that
\[
	\Ex \bigl( \|\Yvecnhx\|^2 \bigr) \ = \ O \bigl( (nh^d)^{-1} \bigr)
	\quad\text{and}\quad
	\Bvechx \ = \ \bigl( O(h^2), O(h^3), O(h^4) \bigr)
\]
uniformly in $\x \in \R^d$.
\end{rem}

\paragraph{Rates of convergence.} Theorems~\ref{thm:from.f.to.ell} and \ref{thm:KD_f} imply the following rates of convergence for estimators of $g = f, \ell$ and its derivatives:
\begin{align*}
	\hat{g}_{n,h}^{\mathtt{K}}(\x) \
	&= \ g(\x) + O_p(n^{-1/2} h^{-d/2}) + O(h^2) \\
	&= \ g(\x) + \begin{cases}
		O_p(n^{-2/(d+4)})
			& \text{if} \ h = C n^{-1/(d+4)} , \\
		O_p(n^{-4/(d+8)}) + O(n^{-2/(d+8)})
			& \text{if} \ h = C n^{-1/(d+8)} ,
		\end{cases} \\
	\hat{Dg}_{n,h}^{\mathtt{K}}(\x) \
	&= \ Dg(\x) + O_p(n^{-1/2}h^{-d/2-1}) + O(h^2) \\
	&= \ Dg(\x) + \begin{cases}
		O_p(n^{-2/(d+6)})
			& \text{if} \ h = C n^{-1/(d+6)} , \\
		O_p(n^{-3/(d+8)}) + O(n^{-2/(d+8)})
			& \text{if} \ h = C n^{-1/(d+8)} ,
		\end{cases} \\
	\hat{D^2g}_{n,h}^{\mathtt{K}}(\x) \
	&= \ D^2g(\x) + O_p(n^{-1/2} h^{-d/2-2}) + O(h^2) \\
	&= \ D^2 g(\x) + O_p(n^{-2/(d+8)}) \quad\text{if} \ h = C n^{-1/(d+8)} ,
\end{align*}
where $C > 0$.

\subsection{Local Log-Likelihood Estimation}
\label{sec:LL-asym}

For local log-likelihood estimation, we need to impose an exponential moment condition on the kernel $K$:
\begin{itemize}[leftmargin=.75in]
\item[(K3$^{\mathtt{L}}$)] \ For some $\epsilon(K) \in (0,\infty]$,
\[
	\int K(\z) \exp(\epsilon\|\z\|^2/2) \, d\z < \infty
	\quad\text{if} \ \eps < \epsilon(K) .
\]
\end{itemize}
In case of a kernel with compact support, this condition is satisfied with $\eps(K) = \infty$. In case of the Gaussian kernel $K = \varphi$, (K3$^{\mathtt{L}}$) holds true with $\eps(K) = 1$.

It follows from condition (K3$^{\mathtt{L}}$) that
\[
	\hat{S}_{n,h}^{\mathtt{L}}(c,\b,\A,\x)
	\ = \ - \Bigl\langle (c, h\b, h^2\A),
			\bigl( \snhx, \bsnhx, \bSnhx \bigr) \Bigr\rangle
		+ \int K(\z) \exp(g_{c,h\b,h^2\A}) \, d\z
\]
is a real-valued, continuously differentiable and strictly convex function of $(c,h\b,h^2\A)$ in the open, convex set
\[
	\Theta
	\ \defeq \ \bigl\{ (\tilde{c},\tilde{\b},\tilde{\A}) \in \Hds :
		\lambda_{\mathtt{max}}(\tilde{\A}) < \eps(K) \bigr\} ,
\]
where $\lambda_{\mathtt{max}}(\A)$ denotes the largest eigenvalue of $\A$, see Lemma~\ref{lem:mgf.etc}. Hence, if $(c,h\b,h^2\A) \in \Theta$, then $(c,\b,\A)$ is the unique minimizer of $\hat{S}_{n,h}^{\mathtt{L}}(\cdot,\cdot,\cdot,\x)$ if and only if the first derivative of $\hat{S}_{n,h}^{\mathtt{L}}(\cdot,\cdot,\cdot,\x)$ at $(c,\b,\A)$ is zero, and this is easily shown to be equivalent to the equation
\[
	\bs{F}(c,h\b,h^2\A) \ = \ \bigl( \snhx, \bsnhx, \bSnhx \bigr) ,
\]
where
\[
	\F(c,\b,\A)
	\ \defeq \ \int K(\z) (1, \z, \z\z^\top) \exp(g_{c,\b,\A}(\z)) \, d\z
\]
for $(c,\b,\A) \in \Theta$. Thus we need a thorough understanding of the mapping $\F$.

\begin{lem}
\label{lem:F.for.LL}
The mapping $\F : \Theta \to \Hds$ is infinitely differentiable. For any $(c,\b,\A) \in \Theta$, the Jacobian operator $D\F(c,\b,\A) : \Hds \to \Hds$ is given by
\[
	D\bs{F}(c,\b,\A) (\beta,\bbeta,\B)
	\ = \ \int K(\z) (1,\z,\z\z^\top)
		\bigl( \beta + \z^\top\bbeta + 2^{-1}\z^\top\B\z \bigr)
		\exp(g_{c,\b,\A}(\z)) \, d\z
\]
and is nonsingular. In particular,
\begin{align*}
	\F(c,\bs{0},\bs{0}) \
	&= \ \exp(c) (1, \bs{0}, \bs{I}_d) , \\
	D\F(c,\bs{0},\bs{0})(\beta,\bbeta,\B) \
	&= \ \exp(c) \bs{J}(\beta, \bbeta, \B) ,
\end{align*}
where $\bs{J} : \Hds \to \Hds$ is the linear operator introduced in Corollary~\ref{cor:LM-Taylor}.

Moreover, $\F(\Theta)$ is an open subset of $\Hds$, $\F : \Theta \to \F(\Theta)$ is bijective, and its inverse mapping $\F^{-1} : \F(\Theta) \to \Theta$ is twice continuously differentiable with Jacobian operator
\[
	D\F^{-1}(c,\b,\A) \ = \ \D\F \bigl( \F^{-1}(c,\b,\A) \bigr)^{-1} .
\]
In particular, $\F(\Theta)$ is an open neighborhood of the set $\bigl\{ \exp(c) (1, \bs{0},\I_d) : c \in \R\bigr\}$.
\end{lem}

Consequently, the local log-likelihood estimator of $\bigl( \ell(\x), hD\ell(\x), h^2 D^2\ell(\x) \bigr)$ is given by
\begin{equation}
\label{eq:LL_ell}
	\bigl( \hat{\ell}_{n,h}^{\mathtt{L}}(\x),
		h\hat{D\ell}_{n,h}^{\mathtt{L}}(\x),
		h^2 \hat{D^2\ell}_{n,h}^{\mathtt{L}} \bigr)
	\ = \ \F^{-1} \bigl( \snhx, \bsnhx, \bSnhx \bigr) ,
\end{equation}
provided that $\bigl( \snhx, \bsnhx, \bSnhx \bigr) \in \F(\Theta)$, and the latter condition is satisfied whenever the triple $\bigl( \snhx, \bsnhx, \bSnhx \bigr)$ is sufficiently close to $\bigl( f(\x), \bs{0}, f(\x) \I_d \bigr)$.

\begin{thm}[Asymptotics for (\texttt{L})]
\label{thm:LL_ell}
The estimators $\hat{f}_{n,h}^{\mathtt{L}}(\x)$, $\hat{Df}_{n,h}^{\mathtt{L}}(\x)$ and $\hat{D^2f}_{n,h}^{\mathtt{L}}(\x)$ in \eqref{eq:LL_ell} are well-defined with asymptotic probability one. They admit a linear expansion (LE) with the following ingredients:
\[
	\Gvec(\z)
	\ \defeq \ f(\x_o)^{-1} K(\z) \bs{J}^{-1}(1,\z,\z\z^\top) .
\]
Moreover,
\begin{align*}
	\Bvec_h(\x)
	= \ &\J^{-1} \bigl( h^4 b(\x_o), h^3 \b(\x_o), h^4 \B(\x_o) \bigr) \\
		&- \ h^4 \bigl( 0, \bs{0},
			\b(\x_o) D\ell(\x_o)^\top + D\ell(\x_o) \b(\x_o)^\top \bigr)
			+ \bigl( o(h^4), o(h^3), o(h^4) \bigr) ,
\end{align*}
where
\begin{align*}
	b(\x_o)
	= \ &24^{-1} \int K(\z) D^4\ell(\x_o; \z) \, d\z , \\
	\b(\x_o)
	= \ &6^{-1} \int K(\z) D^3\ell(\x_o; \z) \, \z \, d\z , \\
	\B(\x_o)
	= \ &24^{-1} \int K(\z) D^4\ell(\x_o;\z) \, \z\z^\top \, d\z \\
		&+ \ 6^{-1} \int K(\z) D^1\ell(\x_o;\z) D^3\ell(\x_o; \z)
			\, \z\z^\top \, d\z .
\end{align*}
\end{thm}

\begin{rem}
Since the first and third compontent of $\J^{-1}(\beta,\bbeta,\B)$ are linear in $(\beta,\B)$ while the second component equals $\bbeta$, we obtain the exponents $\gamma(0) = 4$ and $\gamma(1) = 3$.
\end{rem}

\begin{rem}
By construction of the local moment matching estimators,
\[
	\bs{J}^{-1}(1,\z,\z\z^\top)
	\ = \ \Bigl(
		p_0(\z), \,
		\bigl( p_j(\z) \bigr)_{j=1}^d , \,
		\bigl( p_{jk}(\z) \bigr)_{j,k=1}^d \Bigr)
\]
with the polynomials $p_0$, $p_j$ and $p_{jk}$ given by \eqref{eq:MM2_0}, \eqref{eq:MM2_1} and \eqref{eq:MM2_2}. Consequently, the stochastic parts of $\bigl( \hat{\ell}_{n,h}^{\mathtt{L}}(\x), h\hat{D\ell}_{n,h}^{\mathtt{L}}(\x), h^2\hat{D^2\ell}_{n,h}^{\mathtt{L}}(\x) \bigr)$ and $\bigl( \hat{\ell}_{n,h}^{\mathtt{M}}(\x), h\hat{D\ell}_{n,h}^{\mathtt{M}}(\x), h^2\hat{D^2\ell}_{n,h}^{\mathtt{M}}(\x) \bigr)$ agree up to $o_p \bigl( (nh^d)^{-1/2} \bigr)$. The bias parts have the same order of magnitude, but are different in general.
\end{rem}

\paragraph{Rates of convergence.} Theorem~\ref{thm:LL_ell} implies the following rates of convergence:
\begin{align*}
	\hat{\ell}_{n,h}^{\mathtt{L}}(\x) \
	&= \ \ell(\x) + O_p(n^{-1/2} h^{-d/2}) + O(h^4) , \\
	\hat{D\ell}_{n,h}^{\mathtt{L}}(\x) \
	&= \ D\ell(\x) + O_p(n^{-1/2}h^{-d/2-1}) + O(h^2), \\
	\hat{D^2\ell}_{n,h}^{\mathtt{L}}(\x) \
	&= \ D^2\ell(\x) + O_p(n^{-1/2} h^{-d/2-2}) + O(h^2) .
\end{align*}
In case of
\[
	h \ = \ C n^{-1/(8 + d)}
\]
for some $C > 0$, we obtain
\begin{align*}
	\hat{\ell}_{n,h}^{\mathtt{L}}(\x) \
	&= \ g(\x) + O_p(n^{-4/(8+d)}) , \\
	\hat{D\ell}_{n,h}^{\mathtt{L}}(\x) \
	&= \ D\ell(\x) + O_p(n^{-2/(8+d)}) , \\
	\hat{D^2\ell}_{n,h}^{\mathtt{L}}(\x) \
	&= \ D^2\ell(\x) + O_p(n^{-2/(8+d)}) .
\end{align*}
For the first derivative, one could choose a smaller bandwidth
\[
	h \ = \ C n^{-1/(6+d)}
\]
and then obtain
\[
	\hat{D\ell}_{n,h}^{\mathtt{L}}(\x)
	\ = \ Dg(\x) + O_p(n^{-2/(6+d)}) .
\]

\subsection{Conclusions}

In terms of rates of convergence, local moment matching with a spherically
symmetric kernel and the refined polynomials $p_j$ given by
\eqref{eq:MM3_1} yields the best estimators. Also, the asymptotic behaviour
when estimating $\ell = \log f$ instead of $f$ is similar and can be
related explicitly.

A possible drawback of local moment matching is the possibility that $\hat{f}_{n,h}^{\mathtt{M}}$ may take on negative values, although this happens typically only in the tail regions. This problem does not occur with local log-likelihood estimation, but the rates of convergence are just the same as with simple local moment matching.

Kernel density estimation with nonnegative kernel is the worst method in terms of rates of convergence. An advantage, however, is that one may ignore the bias and re-interpret the estimators $\hat{f}_{n,h}^{\mathtt{K}}$, $\hat{Df}_{n,h}^{\mathtt{K}}$ and $\hat{D^2 f}_{n,h}^{\mathtt{K}}$ as estimators of $f_h$, $Df_h$ and $D^2 f_h$, respectively, where $f_h$ is the convolution of $f$ with $K_h$, i.e.
\[
	f_h(\x) \ = \ \int K_h(\x - \y) f(\y) \, d\y . 
\]
This point of view is adopted by various authors, see for instance \citet{Chaudhuri1999} and \citet{Genovese2014}.

\subsection{Using different bandwidths}
\label{sec:different.bandwidths}

If one is concerned about the bias and would like to use, say, kernel density estimators, one could think about using different bandwitdhs $h(0) \ll h(1) \ll h(2)$, where $h(2) \to 0$ but $n h(0)^d \to \infty$, for instance
\[
	h(j) \ = \ C_j^{} n_{}^{-1/(d + 4 + 2j)}
\]
with constants $C_j > 0$. The results of Section~\ref{sec:DKE-asym} can be adapted to this setting. For the stochastic term, something interesting happens. If $\Gvec(\cdot) = \bigl( g(\cdot), \bs{g}(\cdot), \G(\cdot) \bigr)$ is its main ingredient, the stochastic term $\Yvecnhx$ has the three components
\begin{align*}
	y_{n,h}^{}(\x) \
	&\defeq \ \frac{1}{nh(0)^d} \sum_{i=1}^n
		\bigl[ g(h(0)^{-1}(\X_i - \x)) - \Ex g(h(0)^{-1}(\X_1 - \x)) \bigr] , \\
	\y_{n,h}^{}(\x) \
	&\defeq \ \frac{1}{nh(1)^d} \sum_{i=1}^n
		\bigl[ \bs{g}(h(1)^{-1}(\X_i - \x)) - \Ex \bs{g}(h(1)^{-1}(\X_1 - \x)) \bigr] , \\
	\Y_{\!\!n,h}^{}(\x) \
	&\defeq \ \frac{1}{nh(2)^d} \sum_{i=1}^n
		\bigl[ \G(h(2)^{-1}(\X_i - \x)) - \Ex \G(h(2)^{-1}(\X_1 - \x)) \bigr] .
\end{align*}
By means of moment calculations, see Section~\ref{app:different.bandwidths}, one can show that
\[
	\Bigl( (nh(0)^d)^{1/2} y_{n,h}^{}(\x), \,
		(nh(1)^d)^{1/2} \bs{y}_{n,h}^{}(\x), \,
		(nh(2)^d)^{1/2} \Y_{n,h}^{}(\x) \Bigr)
\]
converges in distribution to a centered Gaussian random vector in $\Hds$ with three \textsl{stochastically independent} components. A heuristic explanation for this phenomenon is that the three components are based mainly on three disjoint portions of the sample: $\hat{f}$ is mainly driven by observations $\X_i$ with $\|\X_i - \x\| \le O(h(0))$. But $\hat{Df}$ depends mainly on observations $\X_i$ with $O(h(0)) \ll \|\X_i - \x\| \le O(h(1))$, whereas $\hat{D^2 f}$ depends mostly on observations $\X_i$ with $O(h(1)) \ll \|\X_i - \x\| \le O(h(2))$.

\section{Local Scoring Rules}
\label{sec:LS}

Local log-likelihood estimation as well as local Hyv\"arinen scores may be viewed as special instances of weighted scoring rules. The latter have been introduced by \citet{HolzmannKlar2017}, and subsequently we present their main results in somewhat more generality such that they apply to our setting of estimating a $\log$-density and its derivatives up to order two.

Let $\PP$ be a class of absolutely continuous probability measures on $\R^d$ represented by their density functions. A \emph{scoring rule} is a map $S: \mathcal{P} \times \R^d \rightarrow \R \cup \{\infty \}$ such that $S(q, p) \defeq \int S(q, \y) p(\y) \, d\y$ exists for all $p, q \in \mathcal{P}$ with $S(p,p) < \infty$. A scoring rule $S$ is \emph{(strictly) proper} with respect to $\mathcal{P}$ if for any $p \in \cP$,
\begin{displaymath}
  S(p,p) \leq S(q,p) \text{ for all } q \in \mathcal{P} \quad (\text{with
    equality if and only if } q = p).
\end{displaymath}

Let $\cW$ be a class of weight functions, that is, it consists of bounded,
non-negative, measurable functions on $\R^d$ and let $\GG$ be a class of
non-negative, continuous functions on $\R^d$. The family $\GG$ consists of
local candidates or caricatures of densities in $\PP$. They need not be integrable themselves. A \emph{weighted scoring rule} is a map $\tilde{S}:\GG \times \R^d \times \cW \rightarrow \R \cup \{\infty\}$, such that, for each $w \in \cW$, $\tilde{S}(g,p,w) \defeq \int \tilde{S}(g,\y,w) p(\y) \, d\y$ exists for all $g \in \GG$, $p \in \cP$.

A weighted scoring rule $\tilde{S}$ is \emph{localizing} if, for all $w \in
\cW$ and $g,h \in \GG$, $\tilde{S}(g, \y, w) = \tilde{S}(h, \y, w)$ for
almost all  $\y \in \R^d$ if $g = h$ on $\{w > 0\}$. This implies for each $p \in \cP$, $\tilde{S}(g,p,w) = \tilde{S}(h,p,w)$ if  $g = h$ on $\{w > 0\}$.

A localizing weighted scoring rule is called \emph{locally proper} with
respect to $(\GG, \mathcal{P})$ if, for each $w \in \cW$ and
$p \in \mathcal{P}$, $\tilde{S}(g,p,w) \leq \tilde{S}(h,p,w)$ whenever
$ g,h \in \GG$ and $g = p$ on $\{w > 0\}$. A locally proper weighted
scoring rule $\tilde{S}$ is called \emph{strictly locally proper} with
respect to $(\GG, \cP)$ if it is locally proper with respect to
$(\GG, \cP)$ and for each $w \in \cW$ and $p \in \cP$,
$\tilde{S}(g,p,w) = \tilde{S}(h,p,w)$ for $g, h \in \GG$ with $g = p$ on
$\{w > 0\}$ implies that $h = p$ on $\{w > 0\}$. A localizing weighted
scoring rule is called \emph{proportionally locally proper} with respect to
$(\GG, \cP)$ if, for each $w \in \cW$ and $p \in \cP$,
  $\tilde{S}(g,p,w) \leq \tilde{S}(h,p,w)$ whenever $g, h \in
  \GG$  and $g \propto p$  on $\{w > 0\}$ with equality, if and only if, $h
  \propto p$ on $\{w > 0\}$. Here ``$\propto$'' denotes proportionality of two functions.

The above definitions are a modification of those in \citet{HolzmannKlar2017}. Here, the values of the weight functions are not restricted to $[0,1]$, allowing for kernel functions with arbitrary small bandwidth $h > 0$ as weight functions. Furthermore, the first argument of a weighted scoring rule is not necessarily a probability density function. The idea is to look for a function that is identical or proportional to the underlying probability density function $p \in \cP$ only on the support of the weight function. So, there is no need that the function is a probability density function on $\R^d$ or that it is integrable. 

\citet[Theorem 1]{HolzmannKlar2017} show how to turn proper scoring rules into locally proper scoring rules and strictly proper scoring rules into proportionally locally proper scoring rules. The latter can then be made strictly locally proper; see \citet[Theorem 2]{HolzmannKlar2017}. These statements hold also for the definitions given above and are stated in Theorems \ref{thm:Holzmann1} and \ref{thm:Holzmann2}.

\begin{thm} \label{thm:Holzmann1}
For any weight function $w \in \cW$, we define
\begin{displaymath}
	\GG_w
	\defeq \left\{g \in \GG: 0 < \int w(\y) g(\y) \, d\y < \infty \right\}
	\quad \text{and} \quad
	\tilde{\GG}_w
	\defeq \left\{ g_w := \frac{w \, g}{\int w(\z) g(\z) \, d\z} :
			g \in \GG_w \right\}.
\end{displaymath}
Let $\cP$ be such that $\tilde{\GG}_w \subset \cP$ for all $w \in \cW$ and
$S: \cP \times \R^d \rightarrow \R \cup \{\infty\}$ be a proper scoring rule with respect to $\cP$. Then
\begin{displaymath}
	\tilde{S}: \GG \times \R^d \times \cW \rightarrow
	\R \cup \{\infty\}, \quad\tilde{S}(g, \y, w) \defeq
    \begin{cases}
      w(\y) S(g_w, \y), & \text{if } g \in \GG_w,\\
      \infty, & \text{if } g \notin \GG_w,
    \end{cases}
  \end{displaymath}
  is a locally proper scoring rule with respect to
  $(\GG, \cP)$. Further, if $S$ is strictly proper with
  respect to $\cP$, then $\tilde{S}$ is proportionally locally
  proper with respect to $(\GG, \cP)$.
\end{thm}

\begin{rem}
The assumptions on the scoring rule $S$ can be relaxed. The conclusions of Theorem~\ref{thm:Holzmann1} continue to hold if the restriction of $S$ onto $\tilde{\GG}_w \times \R^d$ is (strictly) proper for any $w \in \WW$.
\end{rem}

\begin{thm}
\label{thm:Holzmann2}
Let $\GG_w$ be as in Theorem \ref{thm:Holzmann1} and $Q:\R_{>0} \times \{0,1\} \rr \R \cup \{\infty\}$ be a scoring rule for a binary outcome, such that for all $\beta \in (0,1]$
\begin{displaymath}
	Q(\beta, \beta) \leq Q(\alpha, \beta)
	\text{ for all } \alpha > 0 \quad \text{(with equality if and only if $\alpha
    = \beta$)}, 
\end{displaymath}
where
\begin{displaymath}
  Q(\alpha, \beta) \defeq \beta Q(\alpha, 1) + (1 - \beta) Q(\alpha, 0).
\end{displaymath}
Then, $S_Q : \GG
\times \R^d \times \cW \rr \R \cup \{\infty\}$,
  \begin{displaymath}
    S_Q(h,\y,w) \defeq
    \begin{cases}
    	\displaystyle
		w(\y) Q \Bigl( m_w^{-1} \int w(\z) h(\z) \, d\z, 1 \Bigr)
		+ \bigl(m_w - w(\y)\bigr) Q \Bigl( m_w^{-1} \int w(\z) h(\z) \, d\z, 0 \Bigr) ,
		& \text{if} \ h \in \GG_w , \\
		\infty ,
		& \text{if} \ h \notin \GG_w ,
    \end{cases}
\end{displaymath}
is locally proper with respect to $(\GG, \cP)$, where $m_w \defeq \sup_{\y \in \R^d} w(\y)$. Furthermore, if $\tilde{S}$ is proportionally locally proper with respect to $(\GG, \cP)$, then
\begin{displaymath}
	\check{S}(h, \y, w) \defeq S_Q(h, \y, w) + \tilde{S}(h, \y, w)
\end{displaymath}
is strictly locally proper with respect to $(\GG, \cP)$.
\end{thm}

\begin{rem}
There are other possible choices for $m_w$. It is only needed, that for
each $w \in \cW$ and $p \in \cP$, we have $\int w(\z) p(\z) \, d\z \leq m_w$.
This ensures that for all $h \in \GG$, such that there exists
$p \in \cP$ with $h = p$ on $\{w > 0\}$, we have $\int w(\z) h(\z) \, d\z < \infty$.

The scoring rule $Q$ in Theorem \ref{thm:Holzmann2} is similar to a strictly proper scoring rule for binary events. The only difference is, that $\alpha$ is not restricted to $(0,1)$. Therefore $S_Q$ can be defined for all $h \in \GG_w$.
\end{rem}

\paragraph{Application to density estimation.}
Suppose $\X_1, \X_2, \ldots, \X_n$ is a random sample drawn from a distribution with density function $f \in \cP$. Let $\tilde{S}$ be a strictly/proportionally locally proper scoring rule with respect to $(\GG, \mathcal{P})$. For some weight function $w \in \cW$,
we are interested in a function $g \in \GG$ with minimal score, so
\begin{displaymath}
	g \ \in \ \argmin_{g \in \GG} \tilde{S}(g,f,w)
	\ = \ \argmin_{g \in \GG} \Ex \tilde{S}(g, \X, w),
\end{displaymath}
where $\X$ is a random vector with distribution according to the density function $f$. Since $f$ is unknown, we replace the distribution of $\X$ with the empirical distribution of $\X_1, \X_2, \ldots, \X_n$. Thus, an estimator for a minimizer of $\tilde{S}(\cdot,f,w)$ is
\begin{displaymath}
	\hat{g} \in \argmin_{g \in \GG}
		\frac{1}{n} \sum_{i=1}^n \tilde{S}(g, \X_i - \x, w).
\end{displaymath}

\paragraph{Local log-quadratic approximation.}
Suppose we are interested in estimating the $\log$-density $\ell = \log f$ and its first and second order derivatives at a particular point $\x \in \{f > 0\}$. Then we use the classes
\begin{equation}
\label{eq:W.and.G.for.x}
	\cW \ \defeq \ \{ K_h(\cdot - \x) : h > 0\}
	\quad\text{and}\quad
	\GG \ = \ \bigl\{ \exp(g_{c,\b,\A}(\cdot - \x) : (c,\b,\A) \in \Hds \bigr\} .
\end{equation}

\paragraph{Local log-likelihood estimation.}
The negative log-likelihood score proposed by \citet{Good1952} is given by
\begin{displaymath}
	S: \mathcal{P} \times \R^d \rightarrow \R \cup \{\infty\}, \quad
	S(p,\y) = - \log p(\y) .
\end{displaymath}
It is a strictly proper scoring rule with respect to any class of
absolutely continuous probability measures $\cP$ such that $\int |\log p(\z)| p(\z) \, d\z < \infty$ for all $p \in \cP$.

If $\GG$ and $\cW$ are such that the conditions in Theorem \ref{thm:Holzmann1} are fulfilled, then the weighted scoring rule
\begin{displaymath}
  \tilde{S} : \GG \times \R^d \times \cW \rightarrow \R \cup \{\infty\},
  \quad \tilde{S}(g, \y, w) \defeq
  \begin{cases}
    - w(\y) \log g_w(\y), & \text{if } g \in \GG_w,\\
    \infty, & \text{if } g \notin \GG_w,
  \end{cases}
\end{displaymath}
is proportionally locally proper with respect to $(\GG, \cP)$.
We can write more explicitly 
\begin{displaymath}
	\tilde{S}(g,\y,w)
	\ = \ - w(\y) \log g(\y)
		+ w(\y) \log\left(\int w(\z) g(\z) \, d\z\right)
		- w(\y) \log w(\y),
\end{displaymath}
which is infinite whenever $g \notin \GG_w$. It is suprising that even with evaluation of the integral $\int w(\z) g(\z) \, d\z$, the score is only proportionally locally proper. However, applying Theorem~\ref{thm:Holzmann2} with the scoring rule 
\begin{displaymath}
	Q: \R_{>0}\times \{0,1\} \rr \R, \quad
	Q(\alpha,z) = -z (\log(\alpha) + 1) + \alpha 
\end{displaymath}
for binary events leads to a strictly locally proper scoring rule. Namely, for $g \in \GG_w$,
\begin{align*}
	S_Q(g,\y,w)
	= &- w(\y) \log \Bigl( m_w^{-1} \int w(\z) g(\z) \, d\z \bigr) - w(\y) \\
		&+ \ w(\y) m_w^{-1} \int w(\z) g(\z) \, d\z
			+ (1 - w(\y) m_w^{-1}) \int w(\z) g(\z) \, d\z \\
	= &- w(\y) \log \left( \int w(\z) g(\z) \, d\z \right)
		+ \int w(\z) g(\z) \, d\z
		+ w(\y) \bigl( \log(m_w) - 1 \bigr),
\end{align*}
and for $g \notin \GG_w$, we have $S_Q(g,\y,w) = \infty$. Hence,
\begin{displaymath}
	\check{S}(g,\y,w) =
		-w(\y) \log g(\y) + \int w(\z) g(\z) \, d\z
		- w(\y) \bigl( 1 + \log(w(\y)/m_w) \bigr)
\end{displaymath}
is a strictly locally proper scoring rule with respect to $(\GG, \cP)$. This score, without the additive term $- w(\y) \log(w(\y)/m_w)$, is known as penalized weighted likelihood rule and studied in detail by \citet{Pelenis2014}. We may neglect all terms not depending on $g \in \GG$ and work with the equivalent score
\begin{displaymath}
    \check{S}(g,\y,w) = -w(\y) \log g(\y) + \int w(\z) g(\z) \, d\z.
\end{displaymath}
For the special families $\cW$ and $\GG$ in \eqref{eq:W.and.G.for.x}, this leads to the local log-likelihood estimator introduced in Section~\ref{sec:LL}.

\paragraph{Local Hyv\"arinen score estimation.}
Now consider the score of \cite{Hyvaerinen2005},
\begin{equation} \label{eq:HS}
  S: \cP \times \R^d \rr \R, \quad
  S(p, \y) \defeq 2^{-1} \lVert D \log p(\y) \rVert^2
  	+ \triangle \log p(\y) ,
\end{equation}
a strictly proper scoring rule with respect to $\mathcal{P}$.  If $\GG$ and $\cW$ are
such that the conditions in Theorem~\ref{thm:Holzmann1} are fulfilled, then
the Hyv\"arinen score leads to a proportionally locally proper scoring rule
$\tilde{S}:\GG \times \R^d \times \cW \rightarrow \R \cup \{\infty\}$,
\begin{align*}
	\tilde{S}(g,\y,w) = w(\y) S(g_w, \y)
	&= 2^{-1} w(\y) \lVert D \log g_w(\y) \rVert^2
		+ w(\y) \triangle \log g_w(\y)\\
	&= 2^{-1} w(\y) \lVert D \log w(\y) + D \log g(\y) \rVert^2
		+ w(\y) \triangle \log g(\y)
		+ w(\y) \triangle \log w(\y)\\
	& = w(\y) S(g,\y) + Dw(\y)^\top D \log g(\y)
		+ w(\y) S(w,\y) .
\end{align*}
Neglecting the term $w(\y) S(w,\y)$ which does not depend on $g \in \GG$, we end up with the weighted scoring rule
\[
	\tilde{S}(g,\y,w)
	= w(\y) S(g,\y) + Dw(\y)^\top D \log g(\y)
\]
which is proportionally locally proper by virtue of Theorem~\ref{thm:Holzmann1}. We refrain from applying Theorem \ref{thm:Holzmann2} to avoid adding an integral term to the score.

For the special families $\cW$ and $\GG$ in \eqref{eq:W.and.G.for.x}, this leads to the local Hyv\"arinen score estimator intruduced in Section~\ref{sec:HS}.

\bibliographystyle{abbrvnat}
\bibliography{../../../Literature/Literature.bib}

\appendix

\section{Appendix: Proofs and Auxiliary Results}
\label{app}

By $\lVert \cdot \rVert_\infty$ we denote the supremum norm of a function
on $\R^d$.

\subsection{Local Moment Matching}
\label{app:MM}

\begin{proof}[\textnormal{\textbf{Proof of Lemma \ref{lem:LM-Taylor}}}]
By Taylor's formula we know that for $\x,\z \in \R^d$ and $h\geq 0$ there exists a $\xi_{\x}(\z) \in [0,h]$ such that
\begin{align*}
    f(\x+h\z)
    &= \sum_{\lvert\bgamma\rvert < L}
    	\frac{h^{\lvert\bgamma\rvert}}{\bgamma!}
			f^{(\bgamma)}(\x) \z^\bgamma
		+ \sum_{\lvert\bgamma\rvert = L} \frac{h^L}{\bgamma!}
			f^{(\bgamma)}(\x + \xi_{\x}(\z) \z) \z^\bgamma\\
    &= \sum_{\lvert\bgamma\rvert \leq L}
    	\frac{h^{\lvert\bgamma\rvert}}{\bgamma!}
			f^{(\bgamma)}(\x) \z^\bgamma
		+ h^L \sum_{\lvert\bgamma\rvert = L}
			\frac{1}{\bgamma!} \bigl(f^{(\bgamma)}(\x+ \xi_{\x}(\z)\z)
				- f^{(\bgamma)}(\x) \bigr)  \z^\bgamma.
\end{align*}
Thus,
\begin{displaymath}
  \int F(\z) f(\x + h\z) \, d\z
  = \sum_{\lvert\bgamma\rvert \leq L}
    \frac{h^{\lvert\bgamma\rvert}}{\bgamma !} f^{(\bgamma)}(\x)
      \int F(\z) \z^{\bgamma} \, d\z
     \\
    + h^L \sum_{\lvert\bgamma\rvert = L} \frac{1}{\bgamma!} R_\bgamma(\x,h) ,
\end{displaymath}
with
\[
  R_\bgamma(\x,h)
  \defeq \int \bigl(f^{(\bgamma)}(\x + \xi_{\x}(\z) \z) - f^{(\bgamma)}(\x) \bigr)
    F(\z) \z^\bgamma \, d\z.
\]
On the one hand, $\lvert R(\x,h)\rvert \le 2 \lVert f^{(\bgamma)}\lVert_\infty C_F$ with $C_F \defeq \int \lvert F(\z)\rvert \lVert\z\rVert^L \, d\z$. On the other hand, for any fixed $B, B' > 0$, uniform continuity of $f^{(\bgamma)}$ on bounded sets implies that
\begin{align*}
  \lvert R(\x,h)\rvert
  &\le 2 \lVert f^{(\bgamma)}\lVert_\infty
        \int\nolimits_{\{\z : \lVert\z\rVert > B'\}}
        	\lvert F(\z)\rvert \lVert\z\rVert^L \,  d\z
      + C_F \sup_{\lVert\x\rVert \le B, \lVert\tilde{\z}\rVert \le hB'}
        \bigl\lvert f^{(\bgamma)}(\x + \tilde{\z}) - f^{(\bgamma)}(\x) \bigr\rvert \\
  &\to 2 \lVert f^{(\bgamma)}\lVert_\infty
      \int\nolimits_{\{\z : \lVert\z\rVert > B'\}}
      	\lvert F(\z)\rvert \lVert\z\rVert^L \,  d\z
      \quad \text{as } h \to 0 ,
\end{align*}
and the latter limit is arbitrarily small for sufficiently large $B'$. The second part of the lemma corresponds to $F(\z) = K(\z) \z^{\alpha}$ with $\lvert F(\z)\rvert \le K(\z) \lVert\z\rVert^L$.
\end{proof}

\begin{proof}[\bf Proof of Lemma~\ref{lem:Inverse.J}]
If we write $\bs{J}(c,\b,\A) = (\tilde{c},\b,\tilde{\A})$, then
\[
	\tilde{c}
	\ = \ c + g
	\quad\text{and}\quad
	\tilde{\A}
	\ = \ c \I_d + \A\odot\M + \mu_{22} \, g \I_d
	\ = \ \tilde{c} \I_d + \A \odot \M + (\mu_{22} - 1) g \I_d
\]
with $g := 2^{-1} \tr(\A)$. In particular, setting $\tilde{\A}_o := \tilde{\A} - \tilde{c}\I_d$,
\[
	\tr(\tilde{\A}_o)
	\ = \ 2^{-1}(\mu_4 - \mu_{22}) \tr(\A) + (\mu_{22} - 1) d g
	\ = \ \eta g .
\]
Indeed, $\eta > 0$, because
\[
	0 \ < \ \int K(\z) (\|\z\|^2 - d)^2 \, d\z
	\ = \ d \mu_4 + d(d-1) \mu_{22} - d^2
	\ = \ d \eta .
\]
Consequently, $g = \eta^{-1} \tr(\tilde{\A}_o)$, and this shows that
\[
	c
	\ = \ \tilde{c} - \eta^{-1} \tr(\tilde{\A}_o) 
\]
and
\[
	\A\odot\M
	\ = \ \tilde{\A}_o
		- (\mu_{22} - 1) \eta^{-1} \tr(\tilde{\A}_o)\I_d ,
\]
whence
\[
	\A
	\ = \ \bigl( \tilde{\A}_o - (\mu_{22} - 1) \eta^{-1} \tr(\tilde{\A}_o) \I_d \bigr)
		\oslash \M .
\]\\[-5ex]
\end{proof}

\begin{proof}[\bf Derivation of the polynomial \eqref{eq:MM3_1}]
If $p_j(\z) = a z_j - b \|\z\|^2 z_j$, then it follows from \eqref{eq:moments.symmetry.K} that for arbitrary $\bgamma \in \N_0^d$ with $|\bgamma| \le 3$,
\[
	\int K(\z) p_j(\z) \z^\bgamma \, d\z \ = \ 0
	\quad\text{unless} \ \bgamma = \e_j \ \text{or} \ \bgamma = \e_j + 2 \e_k
	\ \text{for some} \ 1 \le k \le d.
\]
Hence we have to find $a, b \in \R$ such that
\[
	\int K(\z) p_j(\z) z_j \, d\z \ = \ 1
	\quad\text{and}\quad
	\int K(\z) p_j(\z) z_j z_k^2 \, d\z \ = \ 0
	\quad\text{for} \ 1 \le k \le d .
\]
By symmetry of $K$, it suffices to consider $k = j$ and only one value $k \ne j$. Then we obtain the linear equation system
\begin{align*}
	a - b (\mu_4 + (d-1) \mu_{22}) \
	&\stackrel{!}{=} \ 1 , \\
	a \mu_4 - b (\mu_6 + (d-1) \mu_{42}) \
	&\stackrel{!}{=} \ 0 , \\
	a \mu_{22} - b (2 \mu_{42} + (d-2) \mu_{222}) \
	&\stackrel{!}{=} \ 0 .
\end{align*}
If $\Z \sim K$, then $\Z = R \bs{U}$, where $R = \|\Z\|$ and $\bs{U} \in
\R^d$ are stochastically independent, and $\bs{U}$ is uniformly distributed
on the unit sphere of $\R^d$. Hence $\mu_\bgamma = \Ex(R^{|\bgamma|})
\Ex(\bs{U}^{\bgamma})$. Moreover, in the special case of $K = \varphi$,
well-known identities for the standard Gaussian distribution yield that
$\mu_{22} = \mu_{222} = 1$, $\mu_4 = \mu_{42} = 3$, $\mu_6 = 15$, while
$\Ex(R^4) = d(d+2)$ and $\Ex(R^6) = d(d+2)(d+4)$, and hence
$\Ex(\U^{\bgamma})$ can be computed. Consequently, the left-hand sides of our equation system equal
\begin{align*}
	a - b (\mu_4 + (d-1) \mu_{22}) \
	&= \ a - b \Ex(R^4)/d , \\
	a \mu_4 - b (\mu_6 + (d-1) \mu_{42}) \
	&= \ \frac{3}{d(d+2)} \bigl( a \Ex(R^4) - b \Ex(R^6) \bigr) , \\
	a \mu_{22} - b (2 \mu_{42} + (d-2) \mu_{222}) \
	&= \ \frac{1}{d(d+2)} \bigl( a \Ex(R^4) - b \Ex(R^6) \bigr) .
\end{align*}
Thus, the equation system simplifies to
\begin{align*}
	a - b \Ex(R^4)/d \
	&\stackrel{!}{=} \ 1 , \\
	a \Ex(R^4) - b \Ex(R^6) \	
	&\stackrel{!}{=} \ 0 ,
\end{align*}
and the unique solution is given by
\[
	b \ = \ \Ex(R^4) / \bigl( \Ex(R^6) - \Ex(R^4)^2/d)
	\quad\text{and}\quad
	a \ = \ 1 + b \Ex(R^4)/d .
\]
The denominator of $b$ is strictly positive, because $\Ex(R^2) = d$, so
\[
	\Ex(R^6) - \Ex(R^4)^2/d
	\ = \ \Ex \bigl( R^2 (R^2 - \Ex(R^4)/d)^2 \bigr) \ > \ 0 .
\]\\[-5ex]
\end{proof}

\begin{proof}[\bf Proof of Theorem~\ref{thm:MM_f}]
By construction of the moment matching estimators,
\[
	\bigl( \hat{f}_{n,h}(\x), h \hat{Df}_{n,h}(\x), h^2 \hat{D^2f}_{n,h}(\x) \bigr)
	\ = \ \frac{1}{nh^d} \sum_{i=1}^n \Gvec(h^{-1}(\X_i - \x)) .
\]
Note also that boundedness of $K$ and finiteness of $\int K(\z) \lVert\z\rVert^r \, d\z$ for arbitrary $r > 0$ imply that $\lVert\Gvec\rVert$ and $\lVert\Gvec\rVert^2$ are integrable over $\R^d$. Thus
\[
	\Yvecnhx \ = \ \frac{1}{nh^d} \sum_{i=1}^n
		\bigl[ \Gvec(h^{-1}(\X_i - \x)) - \Ex \Gvec(h^{-1}(\X_1 - \x)) \bigr]
\]
without an additional term $o_p \bigl( (nh^d)^{-1/2} \bigr)$. Consequently, it suffices to analyze
\[
	\Bvechx
	\ \defeq \ h^{-d} \Ex \Gvec(h^{-1}(\X_1 - \x))
		- \bigl( f(\x), h Df(\x), h^2 D^2f(\x) \bigr) .
\]
By construction of the polynomials $p_0$, $p_j$ and $p_{jk}$ and by Lemma~\ref{lem:LM-Taylor},
\begin{align*}
	\Bvechx \
	&= \ \int \Gvec(\z) f(\x + h\z) \, d\z \\
	&= \ \sum_{|\bgamma| = 3, 4} h_{}^{|\bgamma|}
			\frac{f^{(\bgamma)}(\x)}{\bgamma!}
			\int \Gvec(\z) \z^\bgamma \, d\z + o(h^4) \\
	&= \ \sum_{|\bgamma| = 3, 4} h_{}^{|\bgamma|}
			\frac{f^{(\bgamma)}(\x)}{\bgamma!}
			\int \Gvec(\z) \z^\bgamma \, d\z + o(h^4) \\
	&= \ \sum_{|\bgamma| = 3, 4} h_{}^{|\bgamma|}
			\frac{f^{(\bgamma)}(\x)}{\bgamma!}
			\int K(\z) \Bigl( p_0(\z), \bigl( p_j(\z) \bigr)_{j=1}^d,
				\bigl( p_{jk}(\z) \bigr)_{j,k=1}^d \Bigr) \z_{}^\bgamma\, d\z + o(h^4) ,
\end{align*}
Since $p_0(\z)$ and $p_{jk}(\z)$ contain only monomials of order $0$ or $2$,
\[
	\int K(\z) p_0(\z) \z^\bgamma \, d\z
	\ = \ \int K(\z) p_{jk}(\z) \z^\bgamma \, d\z
	\ = \ 0
	\quad\text{whenever} \ |\bgamma| = 3 .
\]
And since $p_j(\z)$ contains only monomials of order $1$ or $3$,
\[
	\int K(\z) p_j(\z) \z^\bgamma \, d\z
	\ = \ 0
	\quad\text{whenever} \ |\bgamma| = 4 .
\]
This shows that in case of $p_j(\z) = z_j$,
\[
	\Bvec_h(\x)
	\ = \ \bigl( h^4 \beta(\x) , h^3 \bs{\beta}(\x) , h^4 \B(\x) \bigr) + o(h^4)
	\ = \ \bigl( h^4 \beta(\x_o) , h^3 \bs{\beta}(\x_o) , h^4 \B(\x_o) \bigr) + o(h^4) ,
\]
where $\beta(\cdot)$, $\bs{\beta}(\cdot)$ and $\B(\cdot)$ are defined as in the theorem, and the last step in the preceding display follows from continuity of the latter three functions. In case of the refined polynomials $p_j(\z)$, their construction implies that
\[
	\Bvec_h(\x)
	\ = \ \bigl( h^4 \beta(\x) , \bs{0} , h^4 \B(\x) \bigr) + o(h^4)
	\ = \ \bigl( h^4 \beta(\x_o) , \bs{0} , h^4 \B(\x_o) \bigr) + o(h^4) .
\]

Let us add two observations about uniform consistency. Note that for arbitrary $\x \in \R^d$,
\begin{align*}
	\Ex \bigl( \bigl\| \Yvecnhx \bigr\|^2 \bigr) \
	&\le \ n^{-1} h^{-2d} \Ex \bigl( \bigl\| \Gvec(h^{-1}(\X_1 - \x)) \bigr\|^2 \bigr) \\
	&= \ (nh^d)^{-1} \int \|\Gvec(\z)\|^2 f(\x + h\z) \, d\z \\
	&\le \ (nh^d)^{-1} \|f\|_\infty \int \|\Gvec(\z)\|^2 \, d\z .
\end{align*}
Moreover, Lemma~\ref{lem:LM-Taylor} implies that
\[
	\Bvechx \ = \ \begin{cases}
	\bigl( O(h^4), O(h^3), O(h^4) \bigr)
		& \text{in general} , \\
	O(h^4)
		& \text{for spherically symmetric $K$ and $p_j$ as in \eqref{eq:MM3_1}} ,
	\end{cases}
\]
uniformly in $\x \in \R^d$.
\end{proof}

\subsection{Kernel Density Estimation}
\label{app:KD}

\begin{proof}[\bf Proof of Theorem~\ref{thm:KD_f}]
By construction of the kernel density estimator,
\[
	\bigl( \hat{f}_{n,h}(\x), h \hat{Df}_{n,h}(\x), h^2 \hat{D^2f}_{n,h}(\x) \bigr)
	\ = \ \frac{1}{nh^d} \sum_{i=1}^n \Gvec(h^{-1}(\X_i - \x)) .
\]
Since $K^{(\balpha)} \in \CC_b^0(\R^d) \cap \LL^1(\R^d)$, both $\lVert\Gvec\rVert$ and $\lVert\Gvec\rVert^2$ are integrable over $\R^d$. Consequently, it suffices to analyze the bias
\[
	\Bvec_h(\x)
	\ \defeq \ h^{-d} \Ex \Gvec(h^{-1}(\X_1 - \x))
		- \bigl( f(\x), h Df(\x), h^2 D^2f(\x) \bigr) .
\]
For any $\balpha \in \N_0^d$ with $|\balpha| \le 2$,
\begin{align*}
	h^{|\balpha|} \Ex \hat{f_{n,h}^{(\balpha)}}(\x) \
	&= \ (-1)^{|\balpha|} h^{-d} \int K^{(\balpha)}(h^{-1}(\y - \x)) f(\y) \, d\y \\
	&= \ (-1)^{|\balpha|} \int K^{(\balpha)}(\z) f(\x + h\z) \, d\z \\
	&= \ h^{|\balpha|} \int K(\z) f^{(\balpha)}(\x + h\z) \, d\z ,
\end{align*}
where the last step follows from $|\balpha|$-fold partial integration combined wih Fubini's theorem. By means of Taylor's formula we obtain the expansion
\begin{align*}
	h^{|\balpha|} \int K(\z) f^{(\balpha)}(\x + h\z) \, d\z \
	&= \ h^{|\balpha|}
		\sum_{|\bgamma| \le 2} h^{|\bgamma|}
			\frac{f^{(\balpha + \bgamma)}(\x)}{\bgamma!}
				\int K(\z) \z^\bgamma \, d\z + o(h^{|\balpha| + 2}) \\
	&= \ h^{|\balpha|} f^{(\balpha)}(\x)
		+ 2^{-1} h^{|\balpha| + 2} \sum_{m=1}^d f^{(\balpha + 2\e_m)}(\x)
		+ o(h^{|\balpha| + 2}) ,
\end{align*}
because $\int K(\z) \z^\bgamma \, d\z = \one\{\bgamma = \bs{0}\} + \sum_{m=1}^d \one\{\bgamma = 2\e_m\}$. From this representation and continuity of the partial derivatives of $f$ involved here, one can easily derive the representation of $\Bvec_h^{}(\x)$.
\end{proof}

\subsection{Local Log-Likelihood Estimation}
\label{app:LL}

\begin{proof}[\bf Mininizing $\hat{S}_{n,h}^{\texttt{L}}$ in the Gaussian case]
For fixed $\x \in \R^d$, we want to minimize
\begin{align*}
	\hat{S}_{n,h}^{\mathtt{L}}(&c, \b, \A, \x)\\
 = \ &- \frac{1}{n}
	\sum_{i=1}^n K_h(\X_i - \x) \bigl( c + \b^\top (\X_i - \x) +
		2^{-1} (\X_i - \x)^\top \A (\X_i - \x) \bigr) \\
  &+ \ \int K(\z) \exp \bigl( c + h \b^\top \z
  	+ 2^{-1} h^2 \z^\top \A \z \bigr) \, d\z \\
  = \ &- c \snhx - h \b^\top \bsnhx - 2^{-1} \tr \bigl( \A \bSnhx \bigr)
  		+ \int K(\z)
			\exp \bigl( c + h \b^\top \z + 2^{-1} h^2 \z^\top \A \z \bigr) \, d\z
\end{align*}
in $(c,\b,\A)$. In case of $K = \varphi$, the latter integral equals $\infty$ if $h^2 \lambda_{\mathtt{max}}(\A) \ge 1$. Otherwise, $\A_h^{} \defeq \I_d - h^2\A$ is positive definite, and
\begin{align*}
	\int & K(\z) \exp \bigl( c + h \b^\top \z + 2^{-1} h^2 \z^\top \A \z \bigr) \, d\z \\
	&= \ (2\pi)^{-d/2} \int
		\exp \bigl( - 2^{-1} (\z - h\A_h^{-1} \b)^\top \A_h^{} (\z - h\A_h^{-1}\b) \bigr)
				\, d\z \,
		\cdot \exp \bigl( c + 2^{-1} h^2 \b^\top\A_h^{-1}\b \bigr) \\
	&= \ (2\pi)^{-d/2}
		\int \exp \bigl( - 2^{-1} \y^\top \A_h^{} \y \bigr) \, d\y \,
		\cdot \exp \bigl( c + 2^{-1} h^2 \b^\top\A_h^{-1}\b \bigr) \\
	&= \ \det(\A_h^{})^{-1/2} \exp \bigl( c + 2^{-1} h^2 \b^\top\A_h^{-1}\b \bigr) .
\end{align*}
Consequently,
\[
	\hat{S}_{n,h}^{\mathtt{L}}(c, \b, \A, \x)
	= - c \snhx - h \b^\top \bsnhx - 2^{-1} h^2 \tr \bigl( \A \bSnhx \bigr)
		+ \ \det(\A_h^{})^{-1/2} \exp \bigl( c + 2^{-1} h^2 \b^\top \A_h^{-1} \b \bigr) .
\]
This is strictly convex in $c$, and since
\begin{displaymath}
	\frac{\partial}{\partial c} \hat{S}_{n,h}^{\mathtt{L}}(c, \b, \A, \x)
	= - \snhx
		+ \det(\A_h^{})^{-1/2} \exp \bigl( c + 2^{-1} h^2 \b^\top \A_h^{-1} \b \bigr) ,
\end{displaymath}
the optimal $c$ is given by
\begin{displaymath}
	\hat{c} = \hat{c}(\b,\A)
	= \log \snhx - 2^{-1} h^2 \b^\top \A_h^{-1} \b
		+ 2^{-1} \log \det(\A_h^{}) .
\end{displaymath}
Plugging in $\hat{c}$ leads to the functional 
\begin{align*}
	\hat{S}_{n,h}^{\mathtt{L}}(\hat{c},\b, \A, \x) \
	= \ &\snhx ( 1 - \log \snhx)
		+ 2^{-1} \snhx h^2 \b^\top \A_h^{-1}\b
    	- 2^{-1} \snhx \log \det(\A_h) \\
		&- \ h\b^\top \bsnhx - 2^{-1} h^2 \tr\bigl( \bSnhx \A \bigr) ,
\end{align*}
which is strictly convex and quadratic in $\b$, where
\begin{displaymath}
	\frac{\partial}{\partial \b}
		\hat{S}_{n,h}^{\mathtt{L}}(\hat{c},\hat{\b}, \A, \x)
	= h^2 \snhx \A_h^{-1} \b - h \bsnhx .
\end{displaymath}
hence, the optimal $\b$ is given by
\begin{displaymath}
	\hat{\b} = \hat{\b}(\A) \ = \ h^{-1} \A_h^{} \munhx ,
\end{displaymath}
and plugging in $\hat{\b}$ leads to the functional
\begin{align*}
	\hat{S}_{n,h}^{\mathtt{L}}(\hat{c}, \hat{\b}, \A, \x) \
	= \ &\snhx (1 - \log \snhx) - 2^{-1} \snhx \munhx^\top \A_h^{} \munhx
		- 2^{-1} \log \det(\A_h^{}) \\
		&- \ 2^{-1} \snhx \tr(\bSnhx \A \bigr) \\ 
	= \ & c_{n,h}(\x)
		+ \ 2^{-1} \snhx \bigl( \tr \bigl( \A_h \Sigmanhx \bigr)
			- \log \det(\A_h^{}) \bigr)
\end{align*}
with $c_{n,h}(\x) \defeq \snhx (1 - \log \snhx) - 2^{-1} \tr(\bSnhx)$.
To minimize this with respect to $\A_h^{}$, let $\lambda_1 \leq \lambda_2 \leq \cdots \leq \lambda_d$ be the eigenvalues of the matrix $\A_h \Sigmanhx$. Then
\begin{equation}
\label{eq:Score-LL}
	\hat{S}_{n,h}^{\mathtt{L}}(\hat{c}, \hat{\b}, \A, \x)
	\ = \ c_{n,h}(\x) + 2^{-1} \snhx \sum_{i=1}^d (\lambda_i - \log(\lambda_i) \bigr)
		+ 2^{-1} \log\det(\Sigmanhx) .
\end{equation}
Since the function $x - \log(x)$ is convex with unique minimum at $x = 1$,
\eqref{eq:Score-LL} is minimal if, and only if, $\A_h^{} \Sigmanhx = \I_d$, that means,
\begin{displaymath}
	\hat{\A} \ = \ h^{-2} \bigl( \I_d - \Sigmanhx^{-1} \bigr) ,
\end{displaymath}
provided that $\Sigmanhx$ is nonsingular. (Note that indeed, $h^2 \lambda_{\mathtt{max}}(\hat{\A}) < 1$.) To summarize this, in case of the Gaussian kernel $K = \varphi$ and $\Sigmanhx$ being nonsingular, the local log-likelihood estimators are given by 
\begin{align*}
	\hat{\log f}_{n,h}^{\mathtt{L}}(\x)
	&= \log \snhx - 2^{-1} \munhx^\top \Sigmanhx^{-1} \munhx
		- 2^{-1} \log \det(\Sigmanhx) , \\
	\hat{D \log f}_{n,h}^{\mathtt{L}}(\x)
	&= h^{-1} \Sigmanhx^{-1} \munhx,\\
	\hat{D^2 \log f}_{n,h}^{\mathtt{L}}(\x)
	&= h^{-2} \bigl( \I_d - \Sigmanhx^{-1} \bigr).
\end{align*}\\[-5ex]
\end{proof}

\begin{proof}[\bf Proof of Lemma~\ref{lem:F.for.LL}]
With the measure $\mu(d\z) \defeq K(\z) \, d\z$ and the mapping $\vec{T} : \R^d \to \Hds$, $\vec{T}(\z) := (1,\z,\z\z^\top)$, the mapping $\F : \Theta \to \Hds$ is just the gradient of the function $M : \Hds \to \R$ with
\[
	M(\Hvec) \ := \ \int \exp \bigl( \langle \Hvec, \vec{T}(\z) \rangle \bigr) \, \mu(d\z) .
\]
Indeed, for $\Hvec \in \Theta$ and $\vec{\Delta} \in \Hds$,
\begin{align*}
	M(\Hvec + \vec{\Delta}) - M(\Hvec) \
	&= \ \int \langle \vec{\Delta}, \vec{T}(\z) \rangle
		\exp \bigl( \langle \Hvec, \vec{T}(\z) \rangle \bigr) \, \mu(d\z)
		+ O \bigl( \|\vec{\Delta}\|^2 \bigr) \\
	&= \ \Bigl\langle \vec{\Delta},
		\int \vec{T}(\z) \exp \bigl( \langle \Hvec, \vec{T}(\z) \rangle \bigr) \, \mu(d\z)
		\Bigr\rangle + O \bigl( \|\vec{\Delta}\|^2 \bigr) \\
	&= \ \bigl\langle \vec{\Delta}, \F(\Hvec) \bigr\rangle
		+ O \bigl( \|\vec{\Delta}\|^2 \bigr) .
\end{align*}
Furthermore,
\begin{align*}
	\F(\Hvec + \vec{\Delta}) - \F(\Hvec) \
	&= \ \int \vec{T}(\z) \langle \vec{\Delta}, \vec{T}(\z) \rangle
		\exp \bigl( \langle \Hvec, \vec{T}(\z) \rangle \bigr) \, \mu(d\z)
		+ O \bigl( \|\vec{\Delta}\|^2 \bigr) ,
\end{align*}
so the Jacobian operator of $\F$ at $\Hvec \in \Theta$ ist given by
\[
	D\F(\Hvec) \vec{\Delta} \ = \ \int \vec{T}(\z) \langle \vec{\Delta}, \vec{T}(\z) \rangle
		\exp \bigl( \langle \Hvec, \vec{T}(\z) \rangle \bigr) \, \mu(d\z) ,
\]
that means,
\[
	D\F(c,\b,\A)(\beta,\bbeta,\B)
	\ = \ \int K(\z) (1,\z,\z\z^\top)
		\bigl( \beta + \z^\top\bbeta + 2^{-1} \z^\top \B\z \bigr)
		\exp(g_{c,\b,\A}(\z)) \, d\z .
\]

Since we may identify $\Hds$ with $\R^p$, where $p = (1+d)(1 + d/2)$, by choosing some orthonormal basis of $\Hds$, the remaining statements of Lemma~\ref{lem:F.for.LL} are essentially a consequence of Lemma~\ref{lem:mgf.etc}. Indeed, for arbitrary $(c,\b,\A) \in \Hds \setminus \{\bs{0}\}$, the set $\bigl\{ \z \in \R^d : \bigl\langle \vec{T}(\z), (c,\b,\A)\bigr\rangle = 0 \bigr\}$ is either empty or the level set of a nondegenerate affine or quadratic function. Consequently, it has Lebesgue measure $0$.
\end{proof}

For the proof of Theorem~\ref{thm:LL_ell} we need a refinement of Corollary~\ref{cor:LM-Taylor}.

\begin{lem}
\label{lem:LL-Taylor}
The expectation $\bigl( \shx, \bshx, \bShx \bigr)$ of $\bigl( \snhx, \bsnhx, \bSnhx \bigr)$ satisfies
\[
	\bigl( \shx, \bshx, \bShx \bigr)
	\ = \ \F \bigl( \ell(\x), h D\ell(\x), h^2 D^2\ell(\x) \bigr)
		+ \vec{\C}_h(\x)
\]
where
\[
	\vec{\C}_h(\x) \ = \ 
	\Bigl( h^4 c(\x_o) + o(h^4) , \,
			h^3 \bs{c}(\x_o) + o(h^3) , \,
			h^4 \C(\x_o) + o(h^4) \Bigr) ,
\]
with
\begin{align*}
	c(\x_o) \
	&= \ f(\x_o) \int K(\z) \Bigl( \frac{D^4\ell(\x_o;\z)}{24}
			+ \frac{D^1\ell(\x_o;\z) D^3\ell(\x_o;\z)}{6} \Bigr) \, d\z , \\
	\bs{c}(\x_o) \
	&= \ f(\x_o) \int K(\z) \frac{D^3\ell(\x_o;\z)}{6} \z \, d\z , \\
	\C(\x_o) \
	&= \ f(\x_o) \int K(\z) \Bigl( \frac{D^4\ell(\x_o;\z)}{24}
			+ \frac{D^1\ell(\x_o;\z) D^3\ell(\x_o;\z)}{6} \Bigr) \z\z^\top \, d\z .
\end{align*}
\end{lem}

\begin{proof}
The difference $\vec{\C}_h(\x)$ of $\bigl( \shx, \bshx, \bShx \bigr)$ and $\F \bigl( \ell(\x), h D\ell(\x), h^2 D^2\ell(\x) \bigr)$ may be written as
\[
	\vec{\C}_h(\x) \ = \ \int K(\z) (1,\z,\z\z^\top)
		\bigl( f(\x + h\z)
			- f_{\ell(\x),hD\ell(\x),h^2D\ell(\x)}^{}(\z) \bigr) \, d\z	.
\]
Now we fix numbers $0 < \eps_o < \eps(K)$ and $\delta_o > 0$ such that the closed ball around $\x_o$ with radius $\delta_o$ is contained in $\{f > 0\}$. Since $h \to 0$ and $\x \to \x_o$, we know that eventually,
\[
	h \|D\ell(\x)\| \le \eps_o , \quad
	h^2 \lambda_{\mathtt{max}}(D^2\ell(\x)) \le \eps_o
	\quad\text{and}\quad
	\{\x + h\z : \|\z\| \le h^{-1/6}\}
	\subset \{\y : \|\y - \x_o\| \le \delta_o\} \subset \{f > 0\} ,
\]
where $\lambda_{\mathtt{max}}(\A)$ stands for the largest eigenvalue of a symmetric matrix $\A$. From now on we assume that these relations are true.

Note first that
\begin{align*}
	f_{\ell(\x),hD\ell(\x),h^2D\ell(\x)}^{}(\z) \
	&= \ f(\x)
		\exp \bigl( h D\ell(\x)^\top\z + 2^{-1} h^2 \z^\top D^2\ell(\x)\z \bigr) \\
	&\le \ \|f\|_\infty \exp \bigl( \eps_o\|\z\| + \eps_o \|\z\|^2/2 \bigr) .
\end{align*}
Consequently, for any fixed $\eps_* \in (\eps_o, \eps(K))$ and $r \ge 0$, the two integrals
\[
	\int_{\{\z : \|\z\| > h^{-1/6}\}} K(\z) \|\z\|^r
		f(\x + h\z) \, d\z
	\quad\text{and}\quad
	\int_{\{\z : \|\z\| > h^{-1/6}\}} K(\z) \|\z\|^r
			f_{\ell(\x),hD\ell(\x),h^2D\ell(\x)}^{}(\z) \, d\z
\]
are bounded by
\begin{align}
\nonumber
	\|f\|_\infty & \int_{\{\z : \|\z\| > h^{-1/6}\}} K(\z)
		\|\z\|^r \exp \bigl( \eps_o\|\z\| + \eps_o \|\z\|^2/2 \bigr) \, d\z \\
\nonumber
	&\le \ \|f\|_\infty \,
		\sup_{t > h^{-1/6}} \,
			t^r \exp \bigl( \eps_o t - (\eps_* - \eps_o) t^2/2 \bigr)
		\int K(\z)
		\|\z\|^r \exp \bigl( \eps_* \|\z\|^2/2 \bigr) \, d\z \\
\label{ineq:tail.integral}
	&= \ o(h^5) .
\end{align}
In fact, for sufficiently small $h > 0$ and any fixed $0 < \kappa < (\eps_* - \eps_o)/2$,
\[
	\sup_{t > h^{-1/6}} \,
		t^r \exp \bigl( \eps_o t - (\eps_* - \eps_o) t^2/2 \bigr)
	= h^{-r/6} \exp \bigl( \eps_o h^{-1/6} - (\eps_*-\eps_o) h^{-1/3}/2 \bigr)
	= O \bigl( \exp(- \kappa h^{-1/3}) \bigr) .
\]
As a first consequence,
\[
	\vec{\C}_h(\x)
	= \int_{\{\z : \|\z\| \le h^{-1/6}\}}
		K(\z) (1,\z,\z\z^\top) f(\x + h\z)
			\Bigl( 1 - \exp \bigl( g_{\ell(\x),hD\ell(\x),h^2D^2\ell(\x)}^{}(\z)
				- \ell(\x + h\z) \bigr) \Bigr)
		+ o(h^5) .
\]
By Taylor's formula, uniformly in $\z$ with $\|\z\| \le h^{-1/6}$,
\[
	g_{\ell(\x),hD\ell(\x),h^2D^2\ell(\x)}^{}(\z)
		- \ell(\x + h\z) \
	= \ \begin{cases}
		\displaystyle
		- h^3 \frac{D^3\ell(\x;\z)}{6} - h^4 \frac{D^4\ell(\x;\z)}{24}
			+ o(h^4) \|\z\|^4 , \\
		O(h^{5/2}) ,
	\end{cases}
\]
and since $1 - \exp(-t) = t + O(t^2)$ as $t \to 0$,
\[
	1 - \exp \bigl( g_{\ell(\x),hD\ell(\x),h^2D^2\ell(\x)}^{}(\z)
		- \ell(\x + h\z) \bigr)
	\ = \ h^3 \frac{D^3\ell(\x;\z)}{6}
		+ h^4 \frac{D^4\ell(\x,\z)}{24} + o(h^4) \|\z\|^4
		+ O(h^5) .
\]
Consequently,
\begin{align*}
	\vec{\C}_h(\x) \
	&= \ \int_{\{\z : \|\z\| \le h^{-1/6}\}}
		K(\z) (1,\z,\z\z^\top) f(\x + h\z)
			\Bigl( h^3 \frac{D^3\ell(\x;\z)}{6}
				+ h^4 \frac{D^4\ell(\x;\z)}{24} \Bigr) \, d\z
		+ o(h^4) \\
	&= \ \int_{\R^d} K(\z) (1,\z,\z\z^\top) f(\x + h\z)
			\Bigl( h^3 \frac{D^3\ell(\x;\z)}{6}
				+ h^4 \frac{D^4\ell(\x;\z)}{24} \Bigr) \, d\z
		+ o(h^4) ,
\end{align*}
where the last step uses \eqref{ineq:tail.integral}. Finally, since $f(\x + h\z) = f(\x) + O(h)\|\z\| = f(\x) \bigl( 1 + h D^1\ell(\x;\z) \bigr) + O(h^2) \|\z\|^2$, we may write
\begin{align*}
	\vec{\C}_h(\x) \
	&= \ f(\x) \int K(\z) (1,\z,\z\z^\top)
		\Bigl( h^3 \frac{D^3\ell(\x;\z)}{6}
			+ h^4 \Bigl( \frac{D^4\ell(\x;\z)}{24}
				+ \frac{D^1\ell(\x;\z) D^3\ell(\x;\z)}{6} \Bigr) \Bigr) \, d\z
		+ o(h^4) .
\end{align*}
Now we may apply \eqref{eq:moments.symmetry.K} and conclude from continuity of $f$ and $\ell^{(\bgamma)}$ for $|\bgamma| \le 4$ that the three components of $\vec{\C}_h(\x)$ are given by
\begin{align*}
	c_h(\x) \
	&= \ h^4 f(\x) \int K(\z) \Bigl( \frac{D^4\ell(\x;\z)}{24}
				+ \frac{D^1\ell(\x;\z) D^3\ell(\x;\z)}{6} \Bigr) \, d\z
		+ o(h^4) \\
	&= \ h^4 f(\x_o) \int K(\z) \Bigl( \frac{D^4\ell(\x_o;\z)}{24}
				+ \frac{D^1\ell(\x_o;\z) D^3\ell(\x_o;\z)}{6} \Bigr) \, d\z
		+ o(h^4) , \\
	\bs{c}_h(\x) \
	&= \ h^3 f(\x) \int K(\z) \frac{D^3\ell(\x;\z)}{6} \z \, d\z
		+ o(h^4) \\
	&= \ h^3 f(\x_o) \int K(\z) \frac{D^3\ell(\x_o;\z)}{6} \z \, d\z
		+ o(h^3) , \\
	\C_h(\x) \
	&= \ h^4 f(\x) \int K(\z) \Bigl( \frac{D^4\ell(\x;\z)}{24}
				+ \frac{D^1\ell(\x;\z) D^3\ell(\x;\z)}{6} \Bigr) \z\z^\top \, d\z
		+ o(h^4) \\
	&= \ h^4 f(\x_o) \int K(\z) \Bigl( \frac{D^4\ell(\x_o;\z)}{24}
				+ \frac{D^1\ell(\x_o;\z) D^3\ell(\x_o;\z)}{6} \Bigr) \z\z^\top \, d\z
		+ o(h^4) .
\end{align*}\\[-5ex]
\end{proof}

A second ingredient for the proof of Theorem~\ref{thm:LL_ell} is an expansion for $D\F \bigl( \ell(\x), hD^\ell(\x), h^2D^2\ell(\x) \bigr)$.

\begin{lem}
\label{lem:LL-DF}
For arbitrary $(\beta,\bbeta,\B) \in \Hds$,
\begin{align*}
	D\F &\bigl( \ell(\x), hD\ell(\x), h^2 D^2\ell(\x) \bigr) (\beta,\bbeta,\B) \\
	= \ &f(\x) (\beta, \bbeta, \B)
	 	+ h f(\x) \Bigl( D\ell(\x)^\top \bbeta,
			\beta D\ell(\x) + (\B \odot \M) D\ell(\x),
			\bigl( \bbeta D\ell(\x)^\top + D\ell(\x) \bbeta^\top) \odot \M
			\Bigr) \\
		&+ \ O(h^2) \lVert (\beta,\bbeta,\B) \rVert .
\end{align*}
\end{lem}

\begin{proof}
We know that $\F$ is infinitely often differentiable on $\Theta$. Hence, for any compact subset $\KK$ of $\Theta$,
\[
	D\F(\Hvec + \vec{\Delta})\Bvec
	\ = \ DF(\Hvec)\Bvec
		+ \left. \frac{d}{dt} \right\vert_{t=0} D\F(\Hvec + t\vec{\Delta})\Bvec
		+ O \bigl( \lVert\vec{\Delta}\rVert^2 \bigr) \lVert\Bvec\rVert
\]
uniformly in $\Hvec, \Hvec + \vec{\Delta} \in \KK$. With the same notation as in the proof of Lemma~\ref{lem:F.for.LL},
\begin{align*}
	\left. \frac{d}{dt} \right\vert_{t=0} D\F(\Hvec + t\vec{\Delta})\Bvec
	&= \left. \frac{d}{dt} \right\vert_{t=0}
		\int K(\z) \vec{T}(\z) \langle\vec{T}(\z), \Bvec\rangle
		\exp \bigl( \langle \vec{T}(\z),\Hvec + t\vec{\Delta}\rangle \bigr) \, d\z \\
	&= \int K(\z) \vec{T}(\z) \langle\vec{T}(\z), \Bvec\rangle
		\langle\vec{T}(\z), \vec{\Delta}\rangle
		\exp \bigl( \langle \vec{T}(\z),\Hvec\rangle \bigr) \, d\z .
\end{align*}
Applying these findings to $\Hvec = (\ell(\x), \bs{0}, \bs{0})$ and $\vec{\Delta} = \bigl( 0, hD\ell(\x), h^2 D^2\ell(\x) \bigr)$ and $\Bvec = (\beta,\bbeta,\B)$, we conclude that
\begin{align*}
	D\F &\bigl( \ell(\x), hD\ell(\x), h^2 D^2\ell(\x) \bigr) (\beta,\bbeta,\B) \\
	= \ &f(\x) \bs{J}(\beta, \bbeta, \B)
	 	+ f(\x) \int K(\z) (1,\z,\z\z^\top)
			(\beta + \z^\top\bbeta + 2^{-1} \z^\top\B\z) 
			\bigl( h \z^\top D\ell(\x) + 2^{-1} h^2 \z^\top D^2\ell(\x)\z \bigr)
			\, d\z \\
		&+ \ O(h^2) \lVert (\beta,\bbeta,\B) \rVert \\
	= \ &f(\x) \bs{J}(\beta, \bbeta, \B)
	 	+ h f(\x) \int K(\z) (1,\z,\z\z^\top)
			(\beta + \z^\top\bbeta + 2^{-1} \z^\top\B\z) 
			\z^\top D\ell(\x) \, d\z
		\, + O(h^2) \lVert (\beta,\bbeta,\B) \rVert \\
	= \ &f(\x) \bs{J}(\beta, \bbeta, \B) \\
	 	&+ \ h f(\x) \Bigl( D\ell(\x)^\top \bbeta,
			(\beta + 2^{-1}\mu_{22} \tr(\B)) \I_d + (\B \odot \M) \bigr) D\ell(\x), \\
		& \qquad\qquad\qquad
			\bigl( \bbeta D\ell(\x)^\top + D\ell(\x) \bbeta^\top) \odot \M
				+ \mu_{22} \bbeta^\top D\ell(\x) \I_d \Bigr)
		\, + O(h^2) \lVert (\beta,\bbeta,\B) \rVert ,
\end{align*}
because \eqref{eq:moments.symmetry.K} and the definition of $\bs{J}$ imply that
\begin{align*}
	\int K(\z) (\beta + \z^\top\bbeta + 2^{-1} \z^\top\B\z)
		\z^\top D\ell(\x) \, d\z
	&= \int K(\z) \bbeta^\top \z\z^\top D\ell(\x) \, d\z \\
	&= \ D\ell(\x)^\top\bbeta , \\ 
	\int K(\z) (\beta + \z^\top\bbeta + 2^{-1} \z^\top\B\z) \z
		\z^\top D\ell(\x) \, d\z
	&= \int K(\z) (\beta + 2^{-1} \z^\top\B\z) \z\z^\top \, d\z \, D\ell(\x) \\
	&= \bigl( (\beta + 2^{-1}\mu_{22} \tr(\B)) \I_d + (\B \odot \M) \bigr) D\ell(\x) , \\
	\int K(\z) (\beta + \z^\top\bbeta + 2^{-1} \z^\top\B\z) \z\z^\top
		\z^\top D\ell(\x) \, d\z
	&= \int K(\z) \z^\top \bbeta D\ell(\x)^\top\z \, \z\z^\top \, d\z \\
	&= \int K(\z) 2^{-1} \z^\top\A\z \, \z\z^\top \, d\z \\
	&= \A \odot \M + 2^{-1} \mu_{22} \tr(\A) \I_d
\end{align*}
with $\A = \bbeta D\ell(\x)^\top + D\ell(\x) \bbeta^\top$.
\end{proof}

\begin{proof}[\bf Proof of Theorem~\ref{thm:LL_ell}]
By means of Lemma~\ref{lem:LL-Taylor} we can write
\begin{align*}
	\bigl( \snhx, \bsnhx, \bSnhx \bigr)
	&= \bigl( \shx, \bshx, \bShx \bigr) + \Wvec_{n,h}^{}(\x) \\
	&= \F \bigl( \ell(\x), hD\ell(\x), h^2 D^2\ell(\x) \bigr)
		+ \vec{\C}_h(\x) + \Wvec_{n,h}^{}(\x) ,
\end{align*}
where
\[
	\Wvec_{n,h}^{}(\x)
	\defeq \frac{1}{nh^d} \sum_{i=1}^n
		\bigl[ \Hvec(h^{-1}(\X_i - \x)) - \Ex \Hvec(h^{-1}(\X_1 - \x)) \bigr]
	\quad\text{and}\quad
	\Hvec(\z) \defeq K(\z) (1, \z, \z\z^\top) .
\]
Then $\Wvec_{n,h}^{}(\x) = O_p \bigl( (nh^d)^{-1/2} \bigr)$, $\vec{\C}_h(\x) = O(h^3)$ and $\bs{J} \bigl( f(\x), hDf(\x), h^2 D^2f(\x) \bigr) \to f(\x_o) (1, \bs{0}, \I_d) \in \F(\Theta)$. Consequently, the triple $\bigl( \snhx, \bsnhx, \bSnhx \bigr)$ is contained in $\F(\Theta)$ with asymptotic probability one. Moreover, since $D\F \bigl( f(\x_o)(1,\bs{0},\I_d) \bigr) = f(\x_o) \bs{J}$, it follows from Lemma~\ref{lem:delta-method} that
\begin{align*}
	\bigl( \hat{\ell}_{n,h}^{\mathtt{L}}(\x),
		h\hat{D\ell}_{n,h}^{\mathtt{L}}(\x),
		h^2 \hat{D^2\ell}_{n,h}^{\mathtt{L}} \bigr)
	= \ &\F^{-1} \bigl( \snhx, \bsnhx, \bSnhx \bigr) \\
	= \ &\F^{-1} \Bigl( \F \bigl( \ell(\x), hD\ell(\x), h^2 D^2\ell(\x) \bigr)
			+ \vec{\C}_h(\x) \Bigr) \\
		&+ \ \Yvecnhx + o_p \bigl( (nh^d)^{-1/2} \bigr) ,
\end{align*}
where
\[
	\Yvecnhx
	\defeq f(\x_o)^{-1} \bs{J}^{-1}\Wvec_{n,h}^{}(\x)
	= \frac{1}{nh^d} \sum_{i=1}^n
		\bigl[ \Gvec(h^{-1}(\X_i - \x)) - \Ex \Gvec(h^{-1}(\X_1 - \x)) \bigr]
\]
with
\[
	\Gvec(\z) \defeq f(\x_o)^{-1} \bs{J}^{-1} \Hvec(\z)
	= f(\x_o)^{-1} K(\z) \bs{J}^{-1} (1, \z, \z\z^\top) .
\]
This proves the assertion about the stochastic component of the local log-likelihood estimator. It remains to analyze the bias
\[
	\Bvec_h(\x)
	= \F^{-1} \Bigl( \F \bigl( \Lvec_h(\x) \bigr)
			+ \vec{\C}_h(\x) \Bigr)
		- \Lvec_h(\x)
	= \F^{-1} \Bigl( \F \bigl( \Lvec_h(\x) \bigr)
			+ \vec{\C}_h(\x) \Bigr)
		- \F^{-1} \Bigl(
			\F \bigl( \Lvec_h(\x) \bigr) \Bigr) ,
\]
where $\Lvec_h(\x) \defeq \bigl( \ell(\x), hD\ell(\x), h^2 D^2\ell(\x) \bigr)$. Since $\F^{-1}$ is twice continuously differentiable and $\vec{\C}_h(\x) = O(h^3)$, it follows from Taylor's formula that
\[
	\Bvec_h(\x)
	= D\F^{-1} \bigl( \F \bigl( \Lvec_h(\x) \bigr) \bigr) \vec{\C}_h(\x) + O(h^6)
	= D\F \bigl( \Lvec_h(\x) \bigr)^{-1} \vec{\C}_h(\x) + O(h^6) .
\]
We know from Lemma~\ref{lem:LL-DF} that
\[
	D\F \bigl( \Lvec_h(\x) \bigr)
	= D\F \bigl( \ell(\x), \bs{0}, \bs{0} \bigr) + \bs{D}_h(\x) + O(h^2)
	= f(\x) \J + \bs{D}_h(\x) + O(h^2)
\]
with the linear operator $\bs{D}_h(\x) : \Hds \to \Hds$ given by
\begin{align*}
	\bs{D}_h(\x) (\beta, \bbeta, \B)
	= h f(\x) \Bigl( D\ell(\x)^\top \bbeta,
			& \bigl( (\beta + 2^{-1}\mu_{22} \tr(\B)) \I_d + (\B \odot \M) \bigr)
				D\ell(\x), \\
			& \quad
				\bigl( \bbeta D\ell(\x)^\top + D\ell(\x) \bbeta^\top) \odot \M
				+ \mu_{22} D\ell(\x)^\top\bbeta \I_d
			\Bigr) .
\end{align*}
Consequently, $\bs{D}_h(\x) = O(h)$, and via a suitable adaptation of \eqref{eq:vonNeumann},
\begin{align*}
	D\F \bigl( \Lvec_h(\x) \bigr)^{-1} \vec{\C}_h(\x)
	= \ &f(\x)^{-1} \J^{-1} \vec{\C}_h(\x)
		- f(\x)^{-2} \J^{-1} \bs{D}_h(\x) \J^{-1} \vec{\C}_h(\x)
		+ O(h^2) \lVert \vec{\C}_h(\x)\rVert \\
	= \ & \J^{-1} \bigl( f(\x)^{-1} \vec{\C}_h(\x)
		- f(\x)^{-2} \bs{D}_h(\x) \J^{-1} \vec{\C}_h(\x) \bigr)
		+ O(h^5) .
\end{align*}
Note that
\begin{align*}
	\vec{\C}_h(\x)
	&= f(\x) \bigl( h^4 c(\x_o) + o(h^4),
		h^3 \c(\x_o) + o(h^3), h^4 \C(\x_o) + o(h^4) \bigr) \\
	&= f(\x_o) \bigl( h^4 c(\x_o) + o(h^4),
		h^3 \c(\x_o) + o(h^3), h^4 \C(\x_o) + o(h^4) \bigr)
\end{align*}
so
\[
	\J^{-1} \vec{\C}_h(\x) = f(\x_o) \bigl( O(h^4), h^3 \c(\x_o) + o(h^3), O(h^4) \bigr) 
\]
because the first and third component of $\J(\beta,\bbeta,\B)$ are linear functions of $(\beta,\B)$, while the second component just equals $\bbeta$. Together with $\D_h(\x) = O(h)$, this implies that
\begin{align*}
	f&(\x)^{-2} \J^{-1} \bs{D}_h(\x) \J^{-1} \vec{\C}_h(\x) \\
	&= f(\x)^{-2} \J^{-1} \bs{D}_h(\x) \bigl( 0, h^3 \c(\x_o) + o(h^3), \bs{0} \bigr)
		+ O(h^5) \\
	&= h^4 f(\x)^{-1} \J^{-1}
		\Bigl( \c(\x_o)^\top D\ell(\x) ,
			\bs{0} ,
			\bigl( \c(\x_o)D\ell(\x)^\top + D\ell(\x)\c(\x_o)^\top \bigr) \odot \M
				+ \mu_{22} D\ell(\x)^\top\c(\x_o) \I_d \Bigr)
		+ o(h^4) \\
	&= h^4 f(\x)^{-1} \Bigl( 0, \bs{0},
		\c(\x_o)D\ell(\x)^\top + D\ell(\x)\c(\x_o)^\top \Bigr)
		+ o(h^4) \\
	&= h^4 f(\x_o)^{-1} \Bigl( 0, \bs{0},
		\c(\x_o)D\ell(\x)^\top + D\ell(\x)\c(\x_o)^\top \Bigr)
		+ o(h^4) ,
\end{align*}
where the second last step follows from tedious but elementary calculations and Lemma~\ref{lem:Inverse.J}. This proves the asserted representation of $\Bvec_h(\x)$, because $b(\x_o) = f(\x_o)^{-1} c(\x_o)$, $\b(\x_o) = f(\x_o)^{-1} \c(\x_o)$ and $\B(\x_o) = f(\x_o)^{-1} \C(\x_o)$.
\end{proof}

\subsection{Local Hyv\"arinen Score Estimation}
\label{app:LH}

In the proof of \eqref{eq:HS_1} and \eqref{eq:HS_2} we need a basic result about quadratic functions on $\Rddsym$.

\begin{lem}
\label{lem:quadratic.function.Rddsym}
For given matrices $\bSigma \in \Rddsym$ and $\B \in \R^{d\times d}$, where $\bSigma$ is positive definite, let $H : \Rddsym \to \R$ be given by
\[
	H(\A) \ \defeq \ \tr(2^{-1} \A^2 \bSigma) - \tr(\A\B) .
\]
This function $H$ has a unique minimizer $\T_{\bSigma}(\B) \in \Rddsym$. It is the unique solution $\A \in \Rddsym$ of the equation
\[
	\bSigma\A + \A \bSigma \ = \ \B + \B^\top .
\]
If $\B$ is symmetric and $\bSigma\B = \B\bSigma$, then the minimizer is given by
\[
	\T_{\bSigma}(\B) \ = \ \bSigma^{-1} \B .
\]
In general, if $\bSigma = \V \diag(\bs{\lambda}) \V^\top$ for some orthogonal matrix $\V \in \R^{d\times d}$ and a vector $\bs{\lambda} \in (0,\infty)^d$, then the minimizer is given by
\[
	\T_{\bSigma}(\B) \ = \ \V \Bigl( \frac{(\V^\top(\B + \B^\top) \V)_{ij}}
		{\lambda_i + \lambda_j} \Bigr)_{i,j=1}^d \V^\top .
\]
\end{lem}

\begin{proof}
Note first that $H$ is a continuous function with $H(\A) \ge 2^{-1} \lambda_{\rm min}(\bSigma) \lVert\A\rVert_F^2 - \lVert\A\lVert_F \lVert\B\rVert_F \to \infty$ as $\lVert\A\rVert_F \to \infty$. Hence it has at least one minimizer. Furthermore, for $\A, \Delta \in \Rddsym$,
\begin{align*}
	H(\A + \Delta) - H(\A) \
	&= \ 2^{-1} \tr(\Delta^2 \bSigma)
		+ 2^{-1} \tr( \A\Delta \bSigma + \Delta\A\bSigma) - \tr(\Delta\B) \\
	&= \ 2^{-1} \tr(\Delta^2 \bSigma)
	 	+ 2^{-1} \tr \bigl( \Delta( \bSigma\A + \A\bSigma - \B - \B^\top) \bigr) .
\end{align*}
This shows that $H$ is strictly convex, and $\A$ minimizes $H$ if and only if $\tr \bigl( \Delta( \bSigma\A + \A\bSigma - \B - \B^\top) \bigr) = 0$ for all $\Delta \in \Rddsym$, which is equivalent to
\[
	\bSigma\A + \A\bSigma \ = \ \B + \B^\top .
\]
If $\B$ is symmetric and commutes with $\bSigma$, i.e.\ $\bSigma\B = \B\bSigma$, then $\A := \bSigma^{-1}\B = \B\bSigma^{-1}$ defines a symmetric matrix, and
\[
	\bSigma\A + \A\bSigma \ = \ 2\B \ = \ \B + \B^\top .
\]
In general, with the spectral representation $\bSigma = \V \diag(\bs{\lambda}) \V^\top$, 
\begin{align*}
	\bSigma\A + \A\bSigma \
	&= \ \V
		\bigl( \diag(\bs{\lambda}) \V^\top \A \V + \V^\top\A\V \diag(\bs{\lambda}) \bigr)
		\V^\top \\
	&= \ \V \bigl( (\lambda_i + \lambda_j) (\V^\top\A\V)_{ij} \bigr)_{i,j=1}^d \V^\top
\end{align*}
is equal to
\[
	\B + \B^\top
	\ = \ \V \bigl( (\V^\top(\B + \B^\top)\V)_{ij} \bigr)_{i,j=1}^d \V^\top
\]
if and only if
\[
	\V^\top\A\V \ = \ \Bigl( \frac{(\V^\top(\B + \B^\top) \V)_{ij}}
		{\lambda_i + \lambda_j} \Bigr)_{i,j=1}^d ,
\]
and this is equivalent to the asserted formula for $\A = \T_{\bSigma}(\B)$.
\end{proof}

\begin{proof}[\bf Proof of \eqref{eq:HS_1} and \eqref{eq:HS_2}]
Note first that for fixed $\A \in \Rddsym$,
\begin{align*}
	\snhx^{-1} \hat{S}_{n,h}^{\mathtt{H}}(\b,\A,\x) \
	= \ &2^{-1} \lVert\b\rVert^2 + \b^\top \bigl( h \A\munhx + h^{-1} \bqnhx \bigr)
		+ \mathrm{const}_{n,h}^{}(\A,\x) \\
	= \ &2^{-1} \bigl\lVert \b + h\A\munhx + h^{-1} \bqnhx \bigr\rVert^2
		- 2^{-1} \bigl\lVert h\A\munhx + h^{-1} \bqnhx \bigr\rVert^2 \\
		&+ \ \mathrm{const}_{n,h}^{}(\A,\x) \\
	= \ &2^{-1} \bigl\lVert \b + h\A\munhx + h^{-1} \bqnhx \bigr\rVert^2 \\
		&- \ 2^{-1} h^2 \tr \bigl( \A^2 \munhx\munhx^\top \bigr)
		- \bqnhx^\top \A\munhx - 2^{-1} h^{-2} \lVert\bqnhx\rVert^2\\
		&+ \ \mathrm{const}_{n,h}^{}(\A,\x) ,
\end{align*}
and the unique minimizer of this, as a function of $\b$, equals
\[
	\b(\A) \ \defeq \ - h \A\munhx - h^{-1} \bqnhx .
\]
Now our task is to minimize
\begin{align*}
	\snhx^{-1} \hat{S}_{n,h}^{\mathtt{H}}(\b(\A),\A,\x) \
	= \ &2^{-1} h^2 \tr \bigl( \A^2 \Sigmanhx \bigr)
		- \tr(\A (\munhx\bqnhx^\top - \I_d - \bQnhx) \bigr) \\
		&+ \ \mathrm{const}_{n,h}^{}(\x) \\
	= \ &h^2 H(\A) + \mathrm{const}_{n,h}^{}(\x) ,
\end{align*}
where $H$ is defined as in Lemma~\ref{lem:quadratic.function.Rddsym} with $\bSigma = \Sigmanhx$ and $\B = h^{-2} (\munhx\bqnhx^\top - \I_d - \bQnhx)$. Consequently, the optimal $\A$ is given by
\[
	h^2 \A \ = \ \T_{\Sigmanhx}^{}(\munhx\bqnhx^\top - \I_d - \bQnhx)
	\ = \ h^2 \hat{D^2\ell}_{n,h}^{\mathtt{H}}(\x) ,
\]
and then the resulting optimal $\b(\A)$ is given by
\[
	h \b = \ - h^2 \A \munhx - \bqnhx
	\ = \ - h^2 \hat{D^2\ell}_{n,h}^{\mathtt{H}}(\x) \munhx - \bqnhx
	\ = \ h \hat{D\ell}_{n,h}^{\mathtt{H}}(\x) .
\]\\[-5ex]
\end{proof}

\subsection{Linear Expansions}
\label{app:LE}

\begin{proof}[\bf Proof of asymptotic normality of $(nh^d)^{1/2} \Yvecnhx$.]
We apply the multivariate Central Limit Theorem (see Theorem~\ref{thm:Lindeberg}) to
\[
	(nh^d)^{1/2} \Yvecnhx
	\ = \ \sum_{i=1}^n \bigl[ \Zvec_{n,h,i}(\x) - \Ex \Zvec_{n,h,1}(\x) \bigr]
	\quad\text{with} \
	\Zvec_{n,h,i}(\x) \defeq (nh^d)^{-1/2} \Gvec(h^{-1}(\X_i - \x)) .
\]
Since $\Zvec_{n,h,1}(\x), \ldots, \Zvec_{n,h,n}(\x)$ are identically distributed, it suffices to show that for any fixed $\Hvec \in \Hds$,
\begin{equation}
\label{eq:Lindeberg1}
	n \Var \bigl( \bigl\langle\Zvec_{n,h,1}(\x),\Hvec\bigr\rangle\bigr)
	\ \to \ f(\x_o) \bSigma(\Hvec)
\end{equation}
and
\begin{equation}
\label{eq:Lindeberg2}
	n \Ex \bigl( \lVert\Zvec_{n,h,1}(\x)\rVert^2
		\min \bigl\{ \lVert\Zvec_{n,h,1}(\x)\rVert, 1 \bigr\} \bigr)
	\ \to \ 0 .
\end{equation}
The left-hand side of \eqref{eq:Lindeberg1} equals
\begin{align*}
	h^{-d} &\int
			\bigl\langle\Gvec(h^{-1}(\y-\x)),\Hvec\bigr\rangle^2 f(\y) \, d\y
		- \Bigl( h^{-d/2} \int
			\bigl\langle\Gvec(h^{-1}(\y-\x)),\Hvec\bigr\rangle f(\y) \, d\y \Bigr)^2 \\
	&= \ \int \langle\Gvec(\z),\Hvec\rangle^2 f(\x + h\z) \, d\z
		- h^d \Bigl( \int \langle\Gvec(\z),\Hvec\rangle f(\x + h\z) \, d\z \Bigr)^2 .
\end{align*}
Note that the first integrand on the right-hand side conveges pointwise to $f(\x_o) \langle \Gvec(\z), \Hvec\rangle^2$ and is bounded by the integrable function $\|f\|_\infty \langle \Gvec(\z), \Hvec\rangle^2$. The second integrand converges pointwise to $f(\x_o) \langle \Gvec(\z), \Hvec\rangle$ and is bounded by the integrable function $\|f\|_\infty \bigl| \langle \Gvec(\z), \Hvec\rangle \bigr|$. Hence, \eqref{eq:Lindeberg1} follows from dominated convergence. Analogous arguments apply to \eqref{eq:Lindeberg2}: The left-hand side equals
\begin{align*}
	h^{-d} &\int
			\bigl\langle\Gvec(h^{-1}(\y-\x)),\Hvec\bigr\rangle^2
			\min \bigl\{ (nh^d)^{-1/2}
				\bigl|\bigl\langle\Gvec(h^{-1}(\y-\x)),\Hvec\bigr\rangle\bigr|,
					1 \bigr\} f(\y) \, d\y \\
	&= \ \int \langle\Gvec(\z),\Hvec\rangle^2
		\min \bigl\{ (nh^d)^{-1/2} \bigl| \langle\Gvec(\z),\Hvec\rangle \bigr|,
			1 \bigr\} f(\x + h\z) \, d\z .
\end{align*}
This converges to $0$, because the integrand converges pointwise to $0$ and is bounded by the integrable function $\|f\|_\infty \langle \Gvec(\z), \Hvec\rangle^2$.
\end{proof}

\begin{proof}[\bf Proof of Theorem~\ref{thm:from.f.to.ell}]
Note that $\bigl( \ell(\x), h D\ell(\x), h^2 D^2\ell(\x) \bigr) = \bs{F} \bigl( f(\x), h Df(\x), h^2 D^2f(\x) \bigr)$, where $\bs{F} : \Hds \to \Hds$ is given by
\[
	\bs{F}(c,\b,\A) \ := \ \bigl( \log c, c^{-1}\b, c^{-1} \A - c^{-2} \b\b^\top \bigr)
	\quad\text{if} \ c > 0 .
\]
An analogous representation holds true for the moment matching estimators. Note also that $\bs{F}$ is continuously differentiable with Jacobian operator given by
\begin{align}
\nonumber
	D\bs{F}(c,\b,\A) (\beta,\bbeta,\B) \
	\defeq \ &\lim_{t \to 0} \, t^{-1}
		\bigl( \bs{F}(c+t\beta, \b + t\bbeta, \A + t\B)
			- \bs{F}(c,\b,\A) \bigr) \\
\label{eq:DF}
	= \ & c_{}^{-1} \Bigl( \beta , \,
		\bbeta - \beta \b_c,
		\B - \beta \A_c
			- \b_c^{} \bbeta^\top + \bbeta \b_c^\top
			+ 2 \beta \b_c^{} \b_c^\top \Bigr)
\end{align}
with $\b_c \defeq c^{-1}\b$, $\A_c \defeq c^{-1} \A$. But
\[
	\bigl( \hat{f}(\x), h \hat{Df}(\x), h^2 \hat{D^2 f}(\x) \bigr)
	\ = \ \bigl( f(\x), h Df(\x), h^2 D^2 f(\x) \bigr)
		+ \Bvec_h^{}(\x) + \Yvecnhx ,
\]
where $\Bvec_h^{}(\x) = \bigl( O(h^{\gamma(0)}), O(h^{\gamma(1)}), O(h^4) \bigr)$ and $\Yvecnhx = O_p \bigl( (nh^d)^{-1/2} \bigr)$, so
\[
	\bigl( f(\x), h Df(\x), h^2 D^2 f(\x) \bigr) + \Bvec_h(\x)
	\ \to \ \bigl( f(\x_o), \bs{0}, \bs{0} \bigr) .
\]
Note also that
\begin{align*}
	D\bs{F}(f(\x_o),\bs{0},\bs{0}) (\beta,\bbeta,\B) \
	= \ & f(\x_o)^{-1} \bigl( \beta, \bbeta , \B \bigr) .
\end{align*}
Consequently, by a suitable version of the $\delta$-method, see Lemma~\ref{lem:delta-method}, this implies that
\begin{align*}
	\bigl( \hat{\ell}_{n,h}(\x),
		&h \hat{D\ell}_{n,h}(\x), h^2 \hat{D^2\ell}_{n,h}(\x) \bigr) \\
	= \ &\bs{F} \bigl( \bigl( f(\x), h Df(\x), h^2 D^2f(\x) \bigr)
			+ \Bvec_h^{}(\x) \bigr)
		+ f(\x_o)^{-1} \Yvecnhx + o_p \bigl( (nh^d)^{-1/2} \bigr) \\
	= \ &\bs{F} \bigl( \bigl( f(\x), h Df(\x), h^2 D^2f(\x) \bigr)
			+ \Bvec_h^{}(\x) \bigr) \\
		&+ \ \frac{1}{nh^d} \sum_{i=1}^n f(\x_o)^{-1}
			\bigl[ \Gvec(h^{-1}(\X_i - \x)) - \Ex \Gvec(h^{-1}(\X_1 - \x)) \bigr]
		+ o_p \bigl( (nh^d)^{-1/2} \bigr) .
\end{align*}
It remains to analyze the bias 
\begin{align}
\nonumber
	\Bvec_h^{\mathtt{new}}(\x) \
	&= \ \bs{F} \bigl( \bigl( f(\x), h Df(\x), h^2 D^2f(\x) \bigr)
			+ \Bvec_h^{}(\x) \bigr)
		- \bigl( \ell(\x), h D\ell(\x), h^2 D^2\ell(\x) \bigr) \\
\nonumber
	&= \ \bs{F} \bigl( \bigl( f(\x), h Df(\x), h^2 D^2f(\x) \bigr)
			+ \Bvec_h^{}(\x) \bigr)
		- \bs{F} \bigl( f(\x), h Df(\x), h^2 D^2f(\x) \bigr) \\
\label{eq:BiasF}
	&= \ \int_0^1 D\bs{F} \bigl( f(\x), h Df(\x), h^2 D^2f(\x) + t\Bvec_h(\x) \bigr)
		\Bvec_h(\x) \, dt .
\end{align}
Now we plug-in the explicit formula \eqref{eq:DF} with $(\beta,\bbeta,\B) = \Bvec_h^{}(\x)$, that means,
\begin{align*}
	\beta \  &= \ h^{\gamma(0)} \bigl( \beta(\x_o) + o(1) \bigr) , \\
	\bbeta \ &= \ h^{\gamma(1)} \bigl( \bs{\beta}(\x_o) + o(1) \bigr) , \\
	\B \     &= \ h^4 \bigl( \B(\x_o) + o(1) \bigr) ,
\end{align*}
and with $(c,\b,\A) = \bigl( f(\x), h Df(\x), h^2 D^2f(\x) + t\Bvec_h(\x) \bigr)$. Note that uniformly in $t \in [0,1]$,
\begin{align*}
	c \  &= \ f(\x_o) + o(1) , \\
	\b \ &= \ h \bigl( Df(\x_o) + o(1) \bigr) , \\
	\A \ &= \ h^2 \bigl( D^2f(\x_o) + o(1) \bigr) \B(\x_o) + o(1) \bigr) ,
\end{align*}
and thus
\[
	\b_c \ = \ h \bigl( D\ell(\x_o) + o(1) \bigr)
	\quad\text{and}\quad
	\A_c \ = \ h^2 \bigl( D^2\ell(\x_o) + D\ell(\x_o) D\ell(\x_o)^\top + o(1) \bigr) .
\]
Thus, elementary calculations show that the integrand of \eqref{eq:BiasF} equals
\begin{align*}
	f(\x_o)^{-1} \Bigl(
		&h^{\gamma(0)} \beta(\x_o)
			+ o(h^{\gamma(0)}) , \\
		&h^{\gamma(1)} \bs{\beta}(\x_o)
			- h^{\gamma(0)+1} \beta(\x_o) D\ell(\x_o)
			+ o(h^{\min\{\gamma(1),\gamma(0)+1\}}) , \\
		&h^4 \B(\x_o)
			- h^{\gamma(0)+2} \beta(\x_o)
				\bigl( D^2\ell(\x_o) - D\ell(\x_o) D\ell(\x_o)^\top \bigr) \\
		& \qquad - \ h^{\gamma(1) + 1}
				\bigl( \bs{\beta}(\x_o) D\ell(\x_o)^\top
					+ D\ell(\x_o)\bs{\beta}(\x_o)^\top \bigr)
			+ o(h^4)
		\Bigr)
\end{align*}
uniformly in $t \in [0,1]$. From this one can easily deduce the asserted representation of $\Bvec_h^{\mathtt{new}}(\x)$.
\end{proof}

\subsection{Different bandwiths}
\label{app:different.bandwidths}

Asymptotic normality of
\[
	\Bigl( (nh(0)^d)^{d/2} y_{n,h}^{}(\x),
		(nh(1)^d)^{d/2} \y_{n,h}^{}(\x),
		(nh(2)^d)^{d/2} \Y_{\!\!n,h}^{}(\x) \Bigr)
\]
follows from the multivariate CLT with similar arguments as in the previous Section~\ref{app:LE}. To verify that its three components are asymptotically stochastically independent, it suffices to show that for $0 \le j < k \le 2$,
\[
	\Cov \Bigl( h(j)^{-d/2} G_j \bigl( h(j)^{-1}(\X_1 - \x) \bigr),
		h(k)^{-d/2} G_k \bigl( h(k)^{-1}(\X_1 - \x) \bigr) \Bigr)
	\to 0 ,
\]
where
\[
	G_0(\z) \defeq g(\z), \quad
	G_1(\z) \defeq \b^\top \bs{g}(\z)
	\quad\text{and}\quad
	G_2(\z) \defeq \tr(\A \G(\z))
\]
for arbitrary fixed $\b \in \R^d$ and $\A \in \R^{d\times d}_{\mathtt{sym}}$.

On the one hand, for $i=0,1,2$,
\[
	h(i)^{-d/2} \Ex G_i \bigl( h(i)^{-1}(\X_1 - \x) \bigr)
	= h(i)^{d/2} \int G_i(\z) f(\x + h(i)\z) \, d\z
	\to 0 ,
\]
because $h(i) \to 0$, $G_i \in \LL^1(\R^d)$ and $\|f\|_\infty < \infty$. On the other hand,
\begin{align*}
	h(j)^{-d/2} &h(k)^{-d/2}
		\Ex \Bigl( G_j \bigl( h(j)^{-1}(\X_1 - \x) \bigr)
			G_k \bigl( h(k)^{-1}(\X_1 - \x) \bigr) \\
	&= \bigl( h(j)/h(k) \bigr)^{d/2}
		\int G_j(\z) G_k \bigl( (h(j)/h(k)) \bigr) f(\x + h(j)\z) \, d\z
		\ \to 0 ,
\end{align*}
because $G_k$ is bounded.

\subsection{Local Scoring Rules}
\label{app:LS}

\begin{proof}[\textnormal{\textbf{Proof of Theorem \ref{thm:Holzmann1}}}]
  Let $w \in \cW$ be a weight function. For
  $h, g \in \GG$ with $h = g$ on $\{w > 0\}$ either it holds that
  $h,g \in \GG_w$ with $h_w = g_w$ and therefore
  $\tilde{S}(h,\y, w) = \tilde{S}(g, \y, w)$ on $\R^d$, or
  $h,g \in \GG \, \backslash \, \GG_w$ and both scores
  are infinite by definition. Thus, $\tilde{S}$ is a localizing
  weighted scoring rule.

  Let be $p \in \cP$, such that there exists $g \in \GG$
  with $g = p$ on $\{w > 0\}$. For $h \in \GG_w$ we have
  \begin{multline} \label{eq:thm1}
    \tilde{S}(g,p,w) = \int w(\y) S(g_w, \y) p(\y) \, d\y
     = \int S(p_w, \y) p_w(\y) \, d\y \int w(\z) p(\z) \, d\z\\
     \leq \int S(h_w, \y) p_w(\y) \, d\y \int w(\z) p(\z) \, d\z
     = \int w(\y) S(h_w, \y) p(\y) \, d\y = \tilde{S}(h,p,w)
   \end{multline}
   by propriety of $S$ and $g_w = p_w$. For $h \in \GG \, \backslash \,
   \GG_w$ we have
   \begin{equation}\label{eq:thm1b}
     \tilde{S}(g,p,w) \leq \tilde{S}(h,p,w) = \infty
   \end{equation}
   by definition of $\tilde{S}$.  Equation \eqref{eq:thm1} and
   \eqref{eq:thm1b} show that $\tilde{S}$ is locally proper if $S$ is
   proper. Further, if $S$ is strictly proper and
   $g \in \GG$, $p \in \cP$ such that $g \propto p$ on
   $\{w > 0\}$, then is $g_w = p_w$. By Equation \eqref{eq:thm1}
   and \eqref{eq:thm1b} we obtain
   \begin{displaymath}
     \tilde{S}(g,p, w) \leq \tilde{S}(h,p,w) \quad \text{for all} \quad h
     \in \GG
   \end{displaymath}
   with equality if, and only if, $g_w = p_w
   = h_w$ or equivalently $g \propto p \propto h$ on $\{w
   > 0\}$. Therefore, $\tilde{S}$ is proportionally locally proper. The
   statement in the remark holds, because for each $w \in \cW$ we
   use the strict propriety with respect to $\tilde{\GG}_w
   \subset \cP$.
\end{proof}

\begin{proof}[\textnormal{\textbf{Proof of Theorem \ref{thm:Holzmann2}}}]
  The weighted scoring rule $S_Q$ depends on $h \in \GG$ only
  through $\int w(\z) h(\z) \, d\z$, whence $S_Q$ is localizing. To
see that $S_Q$ is proper, let $w \in \cW$ and $p \in
\cP$, such that there exists $g \in \GG$ with $g = p$ on $\{w
> 0\}$. 
For $h \in \GG_w$, we obtain
  \begin{align}
    \frac{S_Q(g, p, w)}{m_w}
    & =  Q\left( \frac{\int w(\y) g(\y) \, d\y}{m_w}, 1
      \right) \frac{\int w(\y) p(\y) \, d\y }{m_w}+ Q \left( \frac{\int w(\y)  g(\y) \,
      d\y}{m_w}, 0\right) \left( 1 - \frac{\int w(\y) p(\y) \, d\y}{m_w}
      \right)\nonumber \\
    &= Q\left( \frac{\int w(\y) g(\y) \, d\y}{m_w}, \frac{\int w(\y) g(\y) \, d\y}{m_w}
      \right)\nonumber \\
    &= Q\left( \frac{\int w(\y) p(\y) \, d\y}{m_w}, \frac{\int w(\y) p(\y) \, d\y}{m_w}
      \right)\nonumber \\
    &\leq Q \left( \frac{\int w(\y) h(\y) \, d\y}{m_w}, \frac{\int w(\y) p(\y) \, d\y}{m_w}
      \right) \label{eq:thm2} \\
    &= Q\left( \frac{\int w(\y) h(\y) \, d\y}{m_w}, 1
      \right) \frac{\int w(\y) p(\y) \, d\y}{m_w} + Q \left( \frac{\int w(\y)  h(\y) \,
      d\y}{m_w}, 0\right) \left( 1 - \frac{\int w(\y) p(\y) \, d\y}{m_w}
      \right) \nonumber \\
    & = \frac{S_Q(h,p,w)}{m_w}. \nonumber
  \end{align}
  For $h \in \GG \backslash \GG_w$ we have $S_Q(g,p,w) <
S_Q(h,q,w) = \infty$.

As a sum of two localizing weighted scoring rules is $\check{S}$ a
localizing weighted scoring rule, too. For each $w \in \cW$,
let be $p \in \cP$, such that there exists $g \in \GG$ with
$g = p$ on $\{w > 0\}$. By the propriety of $S_Q$ and $\tilde{S}$ we have
\begin{displaymath}
  \check{S}(g,p,w) = S_Q(g,p,w) + \tilde{S}(g,p,w) \leq S_Q(h,p, w) +
  \tilde{S}(h,p,w) = \check{S}(h,p,w) \quad \text{for all} \quad h \in \GG.
\end{displaymath}
To see that $\check{S}$ is strictly locally proper, suppose that the
above inequality is indeed an equality. The propriety of $S_Q$ and the
proportional propriety of $\tilde{S}$
imply, that $\tilde{S}(g,p,w) = \tilde{S}(h,p,w)$ and
$S_Q(g,p,w) = S_Q(h,p,w)$. The first identity implies
$h \propto g \propto p$ on $\{w > 0\}$, by the proportional propriety
of $\tilde{S}$. The second identity implies
\begin{displaymath}
  \int w(\z) g(\z) \, d\z = \int w(\z) p(\z) \, d\z = \int w(\z) h(\z) \, d\z,
\end{displaymath}
by the strict propriety of $Q$ and Equation \eqref{eq:thm2}. Both
statements together imply $g = p = h$ on $\{w > 0\}$.
\end{proof}

\subsection{Auxiliary Results}
\label{sec:auxiliary-results}

In this section we collect some results we refer to in the proofs. Most of them are well-known. Unless stated otherwise, asymptotic statements refer to $n \to \infty$.

\begin{thm}[Lindeberg's Multivariate Central Limit Theorem]
\label{thm:Lindeberg}
For $n \in \N$, let $\Y_{\!\!n1}, \Y_{\!\!n2}, \ldots, \Y_{\!\!nn} \in \R^d$ be independent random vectors such that $\Ex(\lVert\Y_{\!\!ni}\rVert^2) < \infty$ for $1 \le i \le n$. Suppose that
\begin{displaymath}
	\bSigma_n \defeq \sum_{i=1}^n \Var(\Y_{\!\!ni})
	\ \rightarrow \ \bSigma
\end{displaymath}
and
\begin{displaymath}
	\sum_{i=1}^n
		\Ex \bigl( \lVert \Y_{\!\!ni} \rVert^2 \min\{\lVert\Y_{\!\!ni}\rVert, 1\} \bigr)
	\ \rightarrow \ 0.
\end{displaymath}
Then, with $\bmu_{ni} := \Ex \Y_{\!\!ni}$,
\begin{displaymath}
	\sum_{i=1}^n (\Y_{\!\!ni} - \bmu_{ni})
	\ \rd \ \cN_d(\bs{0}, \bSigma) .
\end{displaymath}
Furthermore,
\begin{displaymath}
	\Ex \, \Bigl\lVert \sum_{i=1}^n
		(\Y_{\!\!ni} - \bmu_{ni})(\Y_{\!\!ni}-\bmu_{ni})^\top
		- \bSigma_n\Bigr\rVert_F
	\ \to \ 0
	\quad\text{and}\quad
	\Ex \Bigl( \max_{1 \le i \le n} \lVert \Y_{\!\!ni} - \bmu_{ni}\rVert^2 \Bigr)
	\ \to \ 0.
\end{displaymath}
\end{thm}

\begin{proof}
The theorem is well-known in case of $\bmu_{ni} = \bs{0}$ for $1 \le i \le n$; see for instance \citet[Appendix]{DuembgenLinMod2019}. Thus it suffices to show that the conditions stated here imply that
\[
	\sum_{i=1}^n \Ex\bigl( \lVert \Y_{\!\!ni} - \bmu_{ni} \rVert^2
		\min\{\lVert\Y_{ni} - \bmu_{ni}\rVert, 1\} \bigr)
	\rightarrow 0.	
\]
But this is a direct consequence of the next lemma.
\end{proof}

\begin{lem}
\label{lem:Lindeberg.condition}
Let $\Y,\Z \in \R^d$ be random vectors such that $\Ex(\lVert\Y\rVert^2), \Ex(\lVert\Z\rVert^2) < \infty$. Then,
\[
	\Ex \bigl( \lVert\Y+\Z\rVert^2 \min\{\lVert\Y+\Z\rVert,1\} \bigr)
	\ \le \ 8 \Ex \bigl( \lVert\Y\rVert^2 \min\{\lVert\Y\rVert,1\} \bigr)
		+ 8 \Ex \bigl( \lVert\Z\rVert^2 \min\{\lVert\Z\rVert,1\} \bigr)
\]
and, with $\bmu := \Ex\Y$,
\[
	\Ex \bigl( \lVert\Y-\bmu\rVert^2 \min\{\lVert\Y-\bmu\rVert,1\} \bigr)
	\ \le \ 16 \Ex \bigl( \lVert\Y\rVert^2 \min\{\lVert\Y\rVert,1\} \bigr) .
\]
\end{lem}

\begin{proof}
On can easily verify that $h(\y) := \lVert\y\rVert^2 \min\{\lVert\y\rVert,1\} \le \tilde{h}(\y) := 2 \lVert\y\rVert^3/(1 + \lVert\y\rVert)$ for arbitrary $\y \in \R^d$. Since $2x^2/(1 + x)$ is increasing and convex in $x \ge 0$, and since $\y \mapsto \lVert\y\rVert$ ist convex, the function $\tilde{h}$ is convex. Consequently,
\begin{align*}
	\Ex h(\Y + \Z) \
	&\le \ \Ex \tilde{h}(\Y + \Z) \\
	&= \ \Ex \tilde{h}(2^{-1}(2\Y + 2\Z)) \\
	&\le \ 2^{-1} \Ex \tilde{h}(2\Y) + 2^{-1} \Ex \tilde{h}(2\Z) \\
	&= \ 8 \Ex \Bigl( \frac{\lVert\Y\rVert^3}{1 + 2 \lVert\Y\rVert} \Bigr)
		+ 8 \Ex \Bigl( \frac{\lVert\Z\rVert^3}{1 + 2 \lVert\Z\rVert} \Bigr) \\
	&\le \ 8 \Ex \bigl( \lVert\Y\rVert^2 \min\{\lVert\Y\rVert,1\} \bigr)
		+ 8 \Ex \bigl( \lVert\Z\rVert^2 \min\{\lVert\Z\rVert,1\} \bigr) ,
\end{align*}
because $x^3/(1 + 2x) \le x^2\min\{x,1/2\} \le x^2 \min\{x,1\}$ for $x \ge 0$.

In the special case of $\Z = -\Y'$ with an independent copy $\Y'$ of $\Y$, it follows from Jensen's inequality that
\[
	\Ex h(\Y - \bmu)
	\ \le \ \Ex \tilde{h}(\Y - \bmu)
	\ = \ \Ex \tilde{h}(\Ex(\Y - \Y' \,|\, \Y))
	\ \le \ \Ex \Ex \bigl( \tilde{h}(\Y - \Y') \,|\, \Y)
	\ = \ \Ex \tilde{h}(\Y - \Y') ,
\]
and the right hand side is not larger than $16 \Ex \bigl( \lVert\Y\rVert^2 \min\{\lVert\Y\rVert,1\} \bigr)$.
\end{proof}

In connection with linear expansions of our estimatiors we need a variant of the $\delta$-method, see \citet[Theorem 3.8]{Vaart1998}.

\begin{lem}
\label{lem:delta-method}
Let $\bs{F} : \R^p \to \R^q$ such that $\F \in \cC^1(U,\R^q)$ for some open neighborhood $U$ of a point $\x_o \in \R^p$. Consider fixed points $\x_n \in \R^d$ such that $\x_n \to \x_o$ and random vectors $\Y_{\!\!n} \in \R^p$ such that $\Y_{\!\!n} = O_p(r_n)$, where $0 < r_n \to 0$. Then
\[
	\bs{F}(\x_n + \Y_{\!\!n}) \ = \ \bs{F}(\x_n) + D\bs{F}(\x_o)\Y_{\!\!n} + o_p(r_n) ,
\]
where $D\bs{F}(\x_o)$ denotes the Jacobian matrix of $\bs{F}$ at $\x_o$. 
\end{lem}

\begin{proof}
Let $U_n \defeq \{\y \in \R^p : \|\y - \x_o\| < \eps_n\}$ with $\eps_n \defeq \|\x_n - \x_o\| + r_n^{1/2}$. Since $\eps_n \to 0$, there exists an $n_o \in \N$ such that $U_n \subset U$ for all $n \ge n_o$. In case of $n \ge n_o$ and $\|\Y_{\!\!n}\| < r_n^{1/2}$,
\begin{align*}
	\bs{F}(\x_n + \Y_{\!\!n}) \
	&= \ \bs{F}(\x_n) + \int_0^1 D\bs{F}(\x_n + t\Y_{\!\!n})\Y_{\!\!n} \, dt \\
	&= \ \bs{F}(\x_n) + D\bs{F}(\x_o) \Y_{\!\!n} + \bs{R}_n
\end{align*}
with $\bs{R}_n = \int_0^1 \bigl( D\bs{F}(\x_n + t\Y_{\!\!n}) - D\bs{F}(\x_o) \bigr) \Y_{\!\!n} \, dt$. Hence
\[
	\|\bs{R}_n\|
	\ \le \ \int_0^1 \bigl\| D\bs{F}(\x_n + t\Y_{\!\!n}) - D\bs{F}(\x_o) \bigr\|_F \, dt \,
		\|\Y_{\!\!n}\|
	\ \le \ \gamma_n \|\Y_{\!\!n}\|
\]
where
\[
	\gamma_n
	\ \defeq \ \sup_{\y \in U_n} \, \bigl\| D\bs{F}(\y) - D\bs{F}(\x_o) \bigr\|_F
	\ \to \ 0 .
\]
Since $\Pr(\|\Y_{\!\!n}\| \ge r_n^{1/2}) \to 0$, this implies the assertion of the lemma.
\end{proof}

Finally, we recall some useful facts about moment-generating functions.

\begin{lem}
\label{lem:mgf.etc}
Let $\Theta$ be a convex, open subset of $\R^p$, let $\mu$ be a measure on $\R^d$, and let $T:\R^d \to \R^p$ be a measurable mapping such that
\[
	M(\btheta) \ \defeq \ \int \exp(\btheta^\top T(\z)) \, \mu(d\z) \ < \ \infty
	\quad\text{for all} \ \btheta \in \Theta .
\]
Then $M$ is infinitely often differentiable on $\Theta$ with gradient and Hessian matrix given by
\[
	DM(\btheta) \ = \ \int T(\z) \exp(\btheta^\top T(\z)) \, \mu(d\z)
	\quad\text{and}\quad
	D^2M(\btheta) \ = \ \int T(\z)T(\z)^\top \exp(\btheta^\top T(\z)) \, \mu(d\z) ,
\]
respectively. If $\mu(\{\z \in \R^p : \b^\top T(\z) \ne 0\}) > 0$ for any $\b \in \R^p \setminus \{\bs{0}\}$, then the mapping $\F \defeq DM : \Theta \to \R^p$ is injective, the image $\F(\Theta)$ is an open set, too, and the inverse mapping $\G \defeq \F^{-1} : \F(\Theta) \to \Theta$ is twice continuously differentiable with Jacobian matrix
\[
	D\bs{G}(\bs{\eta})
	\ = \ D\bs{F}(\G(\bs{\eta}))^{-1}
	\ = \ D^2M(\G(\bs{\eta}))^{-1} .
\] 
\end{lem}

\begin{proof}
The fact that $M$ is real-analytic and the explicit formulae for its gradient and Hessian matrix are well-known. Maybe less familiar is injectivity of $\F = DM$ in case of $\mu \circ T^{-1}$ not being concentrated on a proper linear subspace of $\R^p$. Indeed, note first that the Jacobian matrix of $\F$ is given by $D\F(\btheta) = D^2M(\btheta)$. This is symmetric and positive definite, because
\[
	\b^\top D\F(\btheta)\b
	\ = \ \int (\b^\top T(\z))^2 \exp(\btheta^\top T(\z) \, \mu(d\z)
	\ > \ 0
\]
for $\b \ne \bs{0}$. Now suppose that $\F(\btheta_0) = \F(\btheta_1)$ for $\btheta_0, \btheta_1 \in \Theta$. Then, with $\btheta_t := (1 - t)\btheta_0 + t\btheta_1 \in \Theta$ for $t \in [0,1]$,
\[
	0 \ = \ (\btheta_1 - \btheta_0)^\top \bigl( \F(\btheta_1) - \F(\btheta_0) \bigr)
	\ = \ (\btheta_1 - \btheta_0)^\top
		\int_0^1 \frac{d}{dt} \F(\btheta_t) \, dt
	\ = \ \int_0^1 (\btheta_1 - \btheta_0)^\top
		D\F(\btheta_t) (\btheta_1 - \btheta_0) \, dt ,
\]
whence $\btheta_1 - \btheta_0 = \bs{0}$.

That $\F(\Theta)$ is open, and that $\G = \F^{-1}$ is also continuously differentiable on $\F(\Theta)$ with $D\G = (D\F \circ \G)^{-1}$ is a direct consequence of the classical inverse function theorem from differential calculus. Since $D\F$ and $\G$ are continuously differentiable, and since the mapping $\A \mapsto \A^{-1}$ is continuously differentiable on the space of nonsingular matrices in $\R^{p\times p}$, the Jacobian $D\G = (\D\F \circ \G)^{-1}$ is continuously differentiable, so $\G$ is twice continuously differentiable.

Continuous differentiability of $\A \mapsto \A^{-1}$ follows essentially from von Neumann's series expansion. For fixed nonsingular $\A \in \R^{p\times p}$,
\begin{equation}
\label{eq:vonNeumann}
	(\A + \Delta)^{-1}
	\ = \ (\I_p + \A^{-1}\Delta)^{-1} \A^{-1}
	\ = \ \sum_{k=0}^\infty (-1)^k (\A^{-1}\Delta)^k \A^{-1}
	\ = \ \A^{-1} - \A^{-1}\Delta\A^{-1} + O(\|\Delta\|_F^2)
\end{equation}
as $\R^{p\times p} \ni \Delta \to \bs{0}$. This expansion shows that $\A \mapsto \A^{-1}$ is continuously differentiable with Jacobian operator $\Delta \mapsto - \A^{-1}\Delta\A^{-1}$ at $\A$.
\end{proof}

\subsection{On the tails of integrable smooth functions}
\label{app:tails.of.smooth.f}

In this section we prove a curious result about the tail behaviour of integrable and smooth functions.

\begin{thm}
\label{thm:smooth.tails}
For an integer $m \ge 0$, let $f \in \LL^1(\R^d) \cap \cC^m(\R^d)$ such that each partial derivative $f^{(\balpha)}$ with $|\balpha| = m$ is uniformly continuous on $\R^d$. Then for any $\bgamma \in \N_0^d$ with $|\bgamma| \le m$,
\[
	|f^{(\bgamma)}(\x)| \to 0
	\quad\text{as} \ \lVert\x\rVert \to \infty .
\]
\end{thm}

Before proving this theorem, we collect a few basic facts about partial and directional derivatives. For an integer $m \ge 1$, the $m$-th directional derivative of a function $f \in \cC^m(\R^d)$ at a point $\x \in \R^d$ in direction $\v \in \R^d$ equals
\[
	D^m f(\x; \v) \defeq \left. \frac{d^m}{dt^m} \right\vert_{t=0}
		f(\x + t\v) .
\]
This may be generalized as follows: For $\v_1,\ldots,\v_m \in \R^d$,
\begin{displaymath}
  D^m f(\x; \v_1, \ldots, \v_m)
  \defeq \left. \frac{\partial^m}{\partial t_1 \cdots \partial
      t_m} \right\vert_{t_1 = \cdots =t_m = 0} f(\x + t_1 \v_1 + \cdots t_m \v_m) ,
\end{displaymath}
so
\[
	D^m f(\x; \v) = D^m f(\x; \underbrace{\v,\ldots,\v}_{m \ \text{times}}) .	
\]
For fixed $\x$, $D^m f(\x; \cdots)$ is a symmetric $m$-linear form on $\R^d$. Indeed, by induction on $m$ one can easily show that for $\v_j = (v_{s,j})_{s=1}^d$,
\[
	D^m f(\x; \v_1, \ldots, \v_m)
	\ = \ \sum_{s(1),\ldots,s(m)=1}^d
		f_{}^{(\bgamma[s(1),\ldots,s(m)])}(\x)
		\prod_{j=1}^m v_{s(j),j}
\]
with $\bgamma[s(1),\ldots,s(m)] = \e_{s(1)} + \cdots + \e_{s(m)}$. This proves $m$-linearity, and symmetry is a consequence of $\bgamma[\cdots]$ being symmetric in its $m$ arguments. For the simple $m$-th directional derivative we have the simplified formula
\begin{equation}
	D^m(\x; \v) \ = \ m! \sum_{|\bgamma| = m}
		\frac{f_{}^{(\bgamma)}(\x)}{\bgamma!} \v_{}^\bgamma .
\end{equation}
Consequently, for any $\v \in \R^d$, one may express $D^m f(\x; \v)$ as a linear combination of the partial derivatives $f^{(\bgamma)}(\x)$ with coefficients depending only on $\v$.

Interestingly, one can also represent any partial derivative $f^{(\bgamma)}(\cdot)$ of order $m$ as a linear combination of finitely many directional derivatives $D^m f(\cdot; \v)$. This follows from the so-called polarization identity for symmetric multilinear forms, see \citet{Mazur1934}. In our setting, for arbitrary $\v_1,\ldots,\v_m \in \R^d$,
\begin{equation}\label{eq:polarization}
	D^mf(\x; \v_1, \ldots, \v_m)
	\ = \ \frac{1}{m!} \sum_{J \subset \{1, \ldots, m\}}
		(-1)^{m - \#J} D^m f \Bigl( \x; \sum_{j \in J} \v_j \Bigr) .
\end{equation}

These considerations have the following consequences:

\begin{lem}
\label{lem:partial.directional}
Let a function $f \in \cC^m(\R^d)$, $m \ge 1$.
\smallskip

\noindent
\textbf{(a)} The following two properties are equivalent:\\[0.5ex]
(a.1) \ For any $\bgamma \in \N_0^d$ with $|\bgamma| = m$, the partial derivative $f^{(\bgamma)}$ is uniformly continuous.\\
(a.2) \ For any fixed $\v \in \R^d$, the $m$-th directional derivative $D^m f(\cdot; \v)$ is uniformly continuous.

\smallskip
\noindent
\textbf{(b)} Moreover, the following two properties are equivalent:\\[0.5ex]
(b.1) \ For any $\bgamma \in \N_0^d$ with $|\bgamma| = m$,
\[
	f^{(\bgamma)}(\x) \ \to \ 0 \quad\text{as} \ \|\x\| \to \infty .
\]
(b.2) \ For any fixed $\v \in \R^d$,
\[
	D^m f(\x; \v) \ \to \ 0
	\quad\text{as} \ \|\x\| \to \infty .
\]
\end{lem}

\begin{proof}[\bf Proof of Theorem~\ref{thm:smooth.tails}]
In case of $m = 0$, we have to show that $f(\x) \to 0$ as $\|\x\| \to \infty$. Suppose that this is wrong. Then there exists a sequence $(\x_n)_{n\in\N}$ in $\R^d$ and a number $\varepsilon > 0$ such that $\lim_{n \to \infty} \|\x_n\| = \infty$ and $|f(\x_n)| \ge 2\eps$ for all $n \in \N$. By uniform continuity of $f$, there exists a number $\delta$ such that $|f(\x)| \ge \varepsilon$ for all $n \in \N$ and any $\x$ in the rectangle $B_n \defeq \x_n + C^d$, where $C \defeq [-\delta,\delta]$. Without loss of generality, let the rectangles $B_n$, $n \in \N$, be pairwise disjoint. Then
\[
	\int |f(\x)| \, d\x
	\ \ge \ \sum_{n\in\N} \int_{B_n} |f(\x)| \, d\x
	\ \ge \ \sum_{n\in\N} \delta^d \varepsilon
	\ = \ \infty ,
\]
a contradiction to $f \in \LL^1(\R^d)$.

Now we consider an integer $m \ge 1$. We first show that
\begin{equation}
\label{eq:vanishing-partial-derivatives}
	f^{(\bgamma)}(\x) \ \to \ 0
	\quad\text{as} \ \|\x\| \to \infty
	\quad\text{whenever} \ |\bgamma| = m .
\end{equation}
According to Lemma~\ref{lem:partial.directional}, this is equivalent to
saying that for any fixed $\v \in \R^d$,
\begin{equation}
\label{eq:vanishing-directional-derivatives}
	D^m f(\x; \v) \ \to \ 0
	\quad\text{as} \ \|\x\| \to \infty .
\end{equation}

\smallskip
\noindent
\textbf{Step 0.} Assume that \eqref{eq:vanishing-directional-derivatives} is
wrong. Then there exist $\v \in \R^d$, $\varepsilon > 0$ and a
series $(\x_n)_{n \in \N}$ in $\R^d$ such that
\begin{displaymath}
  \lVert \x_n \rVert \rightarrow \infty \quad \text{and} \quad \lvert
  D^kf(\x_n; \v) \rvert \geq 2 \varepsilon \quad \text{for all} \ n.
\end{displaymath}
After a suitable linear transformation of the coordinate system we may
assume w.l.o.g. that $\v = \e_1$. By uniform continuity of $D^mf(\x;\e_1)$,
there exists a $\delta > 0$ such that
\begin{displaymath}
	\lvert D^m(\x; \e_1) \rvert \geq \varepsilon
	\quad\text{for all} \ n \in \N \ \text{and} \ \x \in B_n ,
\end{displaymath}
where $B_n = \x_n + C^d$ with $C = [-\delta,\delta]$ as before. Again, we may assume w.l.o.g.\ that these sets $B_n$, $n \in \N$ are pariwise disjoint.

\smallskip
\noindent
\textbf{Step 1 (induction step).} Suppose that for some $k \in \{1,\ldots,m\}$ the following claim is true:
For given numbers $\varepsilon_k > 0$ and $\delta_k \in (0,\delta]$, for all $n \in \N$ and $\y \in C^{d-1}$, there exist $\xi_{n,k}(\y) \in \{-1,1\}$ and an interval $C_{n,k}(\y) \subset C$ of length $2\delta_k$ such that
\begin{displaymath}
	\xi_{n,k}(\y) D^k f \bigl( \x_n + (t,\y); \e_1 \bigr)
	\ \ge \ \varepsilon_k
	\quad\text{for all} \ t \in C_{n,k}(\y) .
\end{displaymath}
Then the same claim is true with $k-1$ in place of $k$, where $\varepsilon_{k-1} = \varepsilon_k \delta_k/2$ and $\delta_{k-1} = \delta_k/4$.

To verify this, note that for $n \in \N$, $\y \in C^{d-1}$ and $u,v \in C_{n,k}(\y)$ with $u < v$,
\begin{align*}
	\xi_{n,k}(\y) &\bigl(
		D^{k-1}f(\x_n + (v, \y), \e_1) - D^{k-1}f(\x_n + (u,\y), \e_1) \bigr) \\
	&= \xi_{n,k}(\y)
		\int_u^v D^kf\bigl( \x_n + (t, \y), \e_1 \bigr) \, dt
	\geq \varepsilon_k (v-u) .
\end{align*}
In particular, let $s$ be the midpoint of $C_{n,k}(\y)$, i.e.\ $C_{n,k}(\y) = [s \pm \delta_k]$. Then, if $\xi_{n,k}(\y) D^{k-1}f(\x_n + (s,\y), \e_1) \geq 0$, we have
\begin{displaymath}
	\xi_{n,k}(\y) D^{k-1}f\bigl( \x_n + (t, \y), \e_1 \bigr)
	\geq \varepsilon_k \delta_k/2
		\quad\text{for} \ t \in [s + \delta_k/2, s + \delta_k] .
\end{displaymath}
And if $\xi_{n,k}(\y) D^{k-1}f(\x_n + (s,\y), \e_1) \leq 0$, then
\begin{displaymath}
	- \xi_{n,k}(\y) D^{k-1}f\bigl( \x_n + (t, \y), \e_1 \bigr)
	\geq \varepsilon_k \delta_k/2
		\quad\text{for} \ t \in [s - \delta_k, s - \delta_k/2] .
\end{displaymath}

\smallskip
\noindent
\textbf{Step~2.}
As shown in Step~0, the claim in Step~1 is true for $k = m$. Then we may repeat Step~2 $m$ times to find that the same claim is true for $k = 0$ with certain numbers $\varepsilon_0 > 0$ and $\delta_0 \in (0,\delta]$, where $D^0f(\x; \e_1) \defeq f(\x)$. But then we may apply Fubini's theorem and conclude that
\[
	\int |f(\x)| \, d\x
	\ \ge \ \sum_{n \in \N}
		\int_{C^{d-1}} \int_{C_{n,0}} \bigl| f(\x_n + (t,\y)) \bigr| \, dt \, d\y
	\ \ge \ \sum_{n \in \N} \delta^{d-1} 2 \delta_0 \varepsilon_0
	\ = \ \infty .
\]
This contradicts our assumption that $f$ is integrable, whence \eqref{eq:vanishing-directional-derivatives} is true.

Having verified \eqref{eq:vanishing-partial-derivatives}, we know that $f \in \LL^1(\R^d) \cap \cC_b^m(\R^d)$. In particular, all partial derivatives $f^{(\bgamma)}$ of order $|\bgamma| = m-1$ are Lipschitz-continuous. Consequently, the previous arguments may be repeated with $m-1$ in place of $m$, and inductively we arrive at the assertion for $m = 0$, which has been treated already.
\end{proof}

\end{document}